\documentclass{SCAEOL}
\numberwithin{equation}{section}
\usepackage{multirow}
\usepackage{endnotes}
\usepackage{algorithmic}
\usepackage{algorithm}
\begin{document}

\Year{2018} %
\Month{April}
\Vol{60} %
\No{1} %
\BeginPage{1} %
\EndPage{XX} %
\AuthorMark{Tao Jie {\it et al.}}

\title{Coordinate Descent algorithm to compute the minimum volume enclosing ellipsoid}{}


\author[1]{Tao Jie}{}
\author[2]{Zhang Wei}{Corresponding author}
\author[3]{Lu Chao}{}

\address[{\rm1}]{School of Business, University of Shanghai for Science and Technology, Yangpu, Shanghai {\rm 200093}, China;}
\address[{\rm2}]{The Logistics Institute, Asia-Pacific, National University of Singapore {\rm 119613}, Singapore;}
\address[{\rm3}]{School of Management, Shanghai University, Baoshan, Shanghai {\rm200444}, China}
\Emails{taojie@usst.edu.cn,
lindelfeel@gmail.com, 06luchao@163.com}\maketitle


 {\begin{center}
\parbox{14.5cm}{\begin{abstract}
The minimum volume enclosing ellipsoid (MVEE) problem is an optimization problem at the root of many practical problems. This paper describes some new properties such as ``algorithmic coordinate-wise smoothness'' of this model and proposes a steepest descent type algorithm, the Coordinate Descent (CD) algorithm, to address the MVEE problem. We prove that not only the function values converges to the optimal value, but also the iteration sequence converges to the optimal solution. The CD algorithm is sublinearly convergent and slightly faster than the other algorithms, especially in cases where the dimension of the data is large. Furthermore, we provide a new interpretation for ways of choosing the coordinate axis of the Frank-Wolfe type algorithm from the perspective of the smoothness of the coordinate axis, i.e., the Khachiyan's algorithm uses the ``Nesterov's Rule'', while the Wolfe - Atwood's algorithm uses the ``Gauss - Southwell's Rule''. Moreover we compare our algorithm with the random coordinate descent method (RCD) in [``Nesterov, Y. Efficiency of coordinate descent methods on huge - scale optimization problems. SIAM Journal on Optimization, 2011, 22(2): 341-362''] and show the RCD algorithm is less efficient than the CD algorithm in computing the MVEE. The numerical tests support our theoretical results.\vspace{-3mm}
\end{abstract}}\end{center}}

 \keywords{Minimum Volume Enclosing Ellipsoid; First-order oracle Algorithm; Coordinate Descent}

 \MSC{90-08, 90C25, 90C30}

\renewcommand{\baselinestretch}{1.2}
\begin{center} \renewcommand{\arraystretch}{1.5}
{\begin{tabular}{lp{0.8\textwidth}} \hline \scriptsize
{\bf Citation:}\!\!\!\!&\scriptsize Tao Jie, Zhang Wei, Lu Chao.  SCIENCE CHINA Mathematics  journal sample. Sci China Math, 2017, 60, doi: 10.1007/s11425-000-0000-0\vspace{1mm}
\\
\hline
\end{tabular}}\end{center}

\baselineskip 11pt\parindent=10.8pt  \wuhao
\section{Introduction}
The minimum volume enclosing ellipsoid (MVEE) (also known as the L\"{o}wner ellipsoid) of an arbitrary given point set $X=\{x_1, \ldots, x_m\}$ is studied in this paper, where $ x_i \in \mathbb{R}^n$. In addition to several traditional applications discussed in \cite{Kumar2005}, recently the MVEE has arisen in a number of applications, including nonlinear support vector machines (\cite{Kumar2011}), bio-informatics (\cite{Fleming2017}), the geometry of differential privacy (\cite{Nikolov2016}), flow feature extraction (\cite{Li2015}), sensor selection (\cite{Joshi2009}), and statistics (\cite{Ahipasaoglu2015}).

F. John first considered the MVEE enclosing a compact set $X \subseteq \mathbb{R}^{n}$. He showed any ellipsoid in $\mathbb{R}^n$ is determined by at most $\frac{n(n+1)}{2}$ points, which lies at the basis of the core set theory. John also showed the existence of the $n$-rounding ellipsoid for an arbitrary point set $X$ (see \cite{John2014}).

Generally, the methods to address the MVEE problem can be categorized into two parts, the Newton-type method and the Frank-Wolfe type method. With respect to the Newton-type method, \cite{Nesterov1994} developed an interior-point algorithm that computes a $(1+\epsilon)-$ approximation to the MVEE in $O(m^{3.5}\ln (\frac{Rm}{\epsilon r}))$ arithmetic operations, where $B_r \subseteq \textrm{conv.hull}(X) \subseteq B_R$ for some Euclidean balls $B_r$ and $B_R$ with radius $0 \leq r \leq R$, $\epsilon$ is the error bound. \cite{Vandenberghe1998}) used the interior point method to solve the max $\det$ problem, which can also be used to solve the MVEE problem. \cite{Sun2004} proposed the ``dual reduced-Newton (DRN)'' algorithm to solve the MVEE problem. This algorithm performs very well for problems with moderately sized samples, while active-set strategies could help in large-scale cases. However, memory problems occur, and the efficiency of the algorithm is reduced when the dimension of the data set grows.

With respect to the Frank-Wolfe method (sometimes called the conditional gradient method), \cite{Khachiyan1996} proposed the notion of $\epsilon$-primal feasibility and computed a $(1+\epsilon)n$ rounding ellipsoid of an arbitrary point set $X$ in $O(mn^2(\frac{1}{\epsilon}+\ln n +\ln \ln m))$ arithmetic operations. Based on the pioneering work of Khachiyan, some progress has been made recently. \cite{Kumar2005} presented an efficient initialization scheme, which could reduce the complexity bound to $O(mn^2(\frac{1}{\epsilon}+\ln n))$. In addition, their algorithm established the existence of an $\epsilon$-core set which could be of great significance in computational geometry. For more details of the core-set, see \cite{Clarkson2010}. Later, \cite{Yildirim2006} applied this algorithm to compute the minimum volume covering ellipsoid of ellipsoids. \cite{Todd2007} modified the work of Kumar and Yildirim with a dropping technique, which is equivalent to the ``away step'' in the Frank-Wolfe method. Coincidentally, this ``away step'' technique was similar to the one adopted by \cite{Atwood1973} in the D-Optimal problem. Therefore this algorithm is also called the Wolfe-Atwood (WA) algorithm. \cite{Ahipasaoglu2008} introduced the $\epsilon$-approximate optimality notion and proved the local linear convergence of the WA algorithm for the MVEE problem, and they showed that the convergence rate is perturbed by a data-dependent constant. Based on this, \cite{Kumar2011} proposed to apply the Frank-Wolfe method with away steps to the support vector classification problem and showed the linear convergence of the algorithm.

The idea of adopting the gradient-type coordinate descent method into the MVEE problem is from two aspects: One is that just as the Frank - Wolfe type algorithm, the coordinate descent method changes only one element of the iteration point $u$ in each iteration, thus the arithmetic operations in each iteration, such as the computation of the inverse and the determinant of a matrix could be simplified via the rank-one modification formula. Therefore the coordinate descent algorithm inherits the advantage of the Frank - Wolfe type algorithm. The other is that the gradient type algorithm is one of the most mature method, and efficient implementation techniques, such as coordinate axis choosing rules can be borrowed. Therefore, in this paper, we integrate the idea of the WA algorithm with the coordinate descent algorithm and propose the coordinate descent (CD) algorithm to address the MVEE problem. {\bf We here stress that although the WA algorithm on the MVEE problem can also be seen as a kind of coordinate descent algorithm, however, the way of choosing axis, the iterative direction and the stepsize are different with the gradient type coordinate descent algorithm. The comparison of the CD algorithm and the WA algorithm is illustrated in section 3.4.}

The rest of this paper is organized as follows: in the second section, we elaborate on the theory and algorithm of the MVEE problem, including the formulation of the model, the existence of the optimal solution, duality and the algorithm to solve the model. In the third section, we introduce the notion of ``algorithmic'' coordinate - wise smoothness and apply this notion to propose the CD algorithm. Moreover, we give new explanation to the WA algorithm and the Khachiyan's algorithm. Specifically, the WA algorithm uses the ``Gauss - Southwell's Rule'' to choose the coordinate axis, while the Khachiyan's algorithm uses the ``Nesterov's Rule'' to choose the coordinate axis. In the fourth section, followed by the idea of Ahipasaoglu (2008) we prove the globally sublinear convergence rate and the locally linear convergence rate of the CD algorithm. In the fifth section, several numerical examples are included, demonstrating that the CD algorithm may be slightly faster than the WA algorithm as the dimension grows. In the sixth section, we discuss more versions of the CD algorithm such as, the CD algorithm with various stepsizes and the random coordinate descent (RCD) algorithm (enlightened by \cite{Nesterov2011}), and compare their computational performance theoretically. In the seventh section, we give a new application of the MVEE to the big data data envelopment analysis problem. Finally, we provide the conclusions and the directions for future research.

\subsection{Notations}

Unless otherwise specified, all norms $\| \cdot \|$ used in this paper are the Euclidean norms. $\| x \|_H =\sqrt{x^T H x}$ represents the $H$-shaped ellipsoidal norm with $H$ being positive definite. $\nabla f(x)$ represents the gradient of function $f$ with respect to the vector $x$, while $\nabla^2 f(x)$ represents the Hessian matrix of function $ f $ with respect to the vector $x$. $e \in \mathbb{R}^n$ is an $n$ dimensional vector in Euclidean space with all the elements being 1. $\mathbb{S}^n$, $\mathbb{S}_{+}^n$, and $\mathbb{S}_{++}^n$ denote the sets of, respectively, the symmetric $n \times n$ matrix, the symmetric positive semidefinite $n \times n$ matrix, and the symmetric positive definite $n \times n$ matrix. $X$ represents the point set of $m$ points, as well as the matrix formed by these $m$ points, i.e., $X=[x_1,\ldots, x_m], x_i \in \mathbb{R}^n, i=1,\ldots,m$. $aff(X)$ represents the affine hull of the set $X$. $I_n$ represents the unit matrix of order $n$. Let $u \in \mathbb{R}^n$, $U$ represents a diagonal matrix of order $n$ formed by the elements of $u$. $\triangle_n$ represents a probability simplex in the $n$ dimensional space, that is, $\triangle_n=\{x | e^T x =1, x \in \mathbb{R}^n_{+}\}$. $\textrm{Vol}(B_n)$ represents the volume of an $n$ - dimensional unit ball whose center is at the origin, that is: $\textrm{Vol}(B_n)=\frac{(\pi)^{\frac{n}{2}}}{\Gamma (\frac{n}{2}+1)}$. ``w.r.t'' is the abbreviation of ``with respect to''.

\section{Theory and Algorithm of MVEE}
\subsection{Volume of $n$-ellipsoid}
An ellipsoid $E(H,c)$ can be expressed as $E(H,c)=\{x | (x-c)^T H (x-c) \leq n\}$, where $c \in \mathbb{R}^n$ and $H \in \mathbb{S}_{++}^n$. Applying the Cholesky factorization to $H \in \mathbb{S}_{++}^n$, we obtain a lower triangular matrix $L$ such that $H=LL^T$. Hence the ellipsoid $E(H,c)$ can also be expressed as $\{x | \| L^T (x-c) \| \leq \sqrt{n}\}$ or $\{x| x=c+(\sqrt{n}L^{-T})z, \| z \| \leq 1\}$, that is, the ellipsoid $E(H,c)$ is the result of the affine transformation $x^{'}=(\sqrt{n}L^{-T})x+c$ of a unit ball centered at the origin. Since the translation does not affect the volume, and the linear transformation $y = \sqrt{n}L^{-T} x$ changes the volume of the unit ball to $\det(\sqrt{n}L^{-T})\textrm{Vol}(B_n)$, it follows that the volume of the ellipsoid in $n$-dimensional space is $\frac{n^{\frac{n}{2}}\textrm{Vol}(B_n)}{\sqrt{\det(H)}}$, that is, the larger $\det(H)$ is, the smaller the volume of $E(H,c)$ is.

\subsection{Model of MVEE}
The MVEE $E(H,0)$ of a centrally symmetric set $X$ can be modeled as an optimization problem (see \cite{Sun2004} for other equivalent forms of MVEE  model)
\[
\max_{H \in \mathbb{S}^n_{++}} \det H, s.t. X \subseteq E(H,0),
\]
or
\[
\min_{H \in \mathbb{S}^n_{++}} -\det H,
\begin{array}{cl}
s.t. & (x_i)^T H (x_i) \leq n, i=1,\ldots, m.\\
\end{array}
\]
Since $-\det H$ is not a convex function, the above optimization problem is difficult to solve. A common method to overcome this difficulty is to convert the objective function into $-\ln \det H$, since $-\ln \det H$ is a convex function (see \cite{Boyd2004}). Thus, the model of $E(H,0)$ can be converted into a convex optimization problem ($P$):
\begin{equation}\tag{$P$}
\begin{array}{cl}
\min_{H \in \mathbb{S}_{++}^n} & f(H)=-\ln \det H,\\
s.t. & (x_i)^T H (x_i) \leq n, i=1,\ldots, m.\\
\end{array}
\end{equation}

\subsection{Existence and Uniqueness of Solutions}
As the objective function of model ($P$) is continuous, \cite{Todd2016} noted that if the set of all the feasible points is compact, the existence of the solution can be proven by the Weierstrass's Theorem. Unfortunately, $H \in \mathbb{S}^n_{++}$ is an open set and this theorem cannot be used directly. \cite{Todd2016} proposed to substitute the constraint $H \in \mathbb{S}^n_{++}$ with $-\ln \det H \leq -\ln \det (\epsilon I_n)$ for a small enough $\epsilon>0$, so that the set of feasible points is compact and the solution of model ($P$) exists. In addition, The uniqueness of the solution is owing to the strict convexity of $-\ln \det H$ on the feasible set.

\subsection{Duality and Optimality Conditions}
The Lagrange function of model ($P$) is: $L(H,u)=-\ln \det H + \sum\limits_{i=1}^m u_i (x_i^T H x_i - n)$, and the dual function is:
\[
q(u)=\min_{H \in \mathbb{S}^n_{++}} L(H,u)=\{\begin{array}{cl}
\ln \det (XUX^T) + n(1-e^T u), & \textrm{if}\  H^{-1}=XUX^T,\\
-\infty, & \textrm{otherwise},\\
\end{array}
\]
where $U$ is a diagonal matrix with $U_{ii}=u_i$. So it follows that the Lagrange dual problem ($D_1$) is:
\begin{equation}\tag{$D_1$}
\begin{array}{cl}
\min_{u} & h(u) = - \ln \det (XUX^T) +n(e^T u - 1),\\
st. & u \geq 0.\\
\end{array}
\end{equation}
Since $\forall u^{'} \geq 0$, there exists a $u \in \mathbb{R}_{+}^n$ with $e^T u=1$ and a $\lambda \geq 0$ such that $u^{'} = \lambda u$. Substituting these into model ($D_1$), we obtain that the optimal value of model ($D_1$) is attained at $\lambda=1$. Therefore, the dual problem can also be written as ($D_2$):
\begin{equation}\tag{$D_2$}
\begin{array}{cl}
\min_u & g(u)= - \ln \det (XUX^T),\\
s.t. & u \in \triangle_m \\
\end{array}
\end{equation}

The optimality conditions, the weak and strong duality can be referred to \cite{Ahipasaoglu2008}, \cite{Todd2016}, which will not be discussed in detail here.

\subsection{MVEE of an arbitrary point set}
\cite{Khachiyan1996} noted that the MVEE of an arbitrary point set $X$ can be obtained by computing the MVEE of a centrally symmetric point set as follows:

(1) Lift up the point $x_i \in X \subseteq \mathbb{R}^n$ into $\mathbb{R}^{n+1}$:
\[
\begin{array}{cc}
\mathcal{A}:& \mathbb{R}^n \rightarrow \mathbb{R}^{n+1}\\
\ & x_i \rightarrow y_i=(x_i,1),\\
\end{array}
\]
and let $Y=\{\pm y_i\}_{i=1}^m$, i.e., $Y$ is a centrally symmetric point set containing $2m$ points.

(2) We can use model ($P$) or model ($D_2$) to compute the MVEE of $Y$. Denoting the optimal solution pair as $(H^{*},u^{*})$, we have:
\[
\begin{array}{cl}
H^{*} &= (YU^{*}Y^{T})^{-1}\\
\ &= (\left(
      \begin{array}{c}
        X \\
        e^T \\
      \end{array}
    \right) U^{*} \left(
      \begin{array}{c}
        X \\
        e^T \\
      \end{array}
    \right)^T)^{-1}\\
\ &= (\left(
        \begin{array}{cc}
          XU^{*}X^T & Xu^{*} \\
          (Xu^{*})^T & 1 \\
        \end{array}
      \right)
)^{-1}\\
\ & = (\left(
          \begin{array}{cc}
            I_n & Xu^{*} \\
            0 & 1 \\
          \end{array}
        \right)
\left(
  \begin{array}{cc}
    XU^{*}X^T-(Xu^{*})(Xu^{*})^T & 0 \\
    0 & 1 \\
  \end{array}
\right)
\left(
  \begin{array}{cc}
    I_n & 0 \\
    (Xu^{*})^T & 1 \\
  \end{array}
\right)
)^{-1}\\
\ & = \left(
        \begin{array}{cc}
          I_n & 0 \\
       -(Xu^{*})^T & 1\\
        \end{array}
      \right)
      \left(
  \begin{array}{cc}
    (XU^{*}X^T-(Xu^{*})(Xu^{*})^T)^{-1} & 0 \\
    0 & 1 \\
  \end{array}
\right)
\left(
          \begin{array}{cc}
            I_n & -Xu^{*} \\
            0 & 1 \\
          \end{array}
        \right).
\end{array}
\]
Let $H_{X}=(XU^{*}X^T-(Xu^{*})(Xu^{*})^T)^{-1}$, $c=Xu^{*}$, we have
\[
(x_i^T, 1) \left(
        \begin{array}{cc}
          I_n & 0 \\
       -c^T & 1\\
        \end{array}
      \right)
      \left(
  \begin{array}{cc}
    H_{X} & 0 \\
    0 & 1 \\
  \end{array}
\right)
\left(
          \begin{array}{cc}
            I_n & c \\
            0 & 1 \\
          \end{array}
        \right)
\left(
          \begin{array}{c}
            x_i\\
            1 \\
          \end{array}
        \right) \leq n+1,
\]
or equivalently
\[
(x_i - c)^T H_{X} (x_i - c) \leq n.
\]
From the above analysis, it can be seen that the MVEE of an arbitrary point set $X$ can be obtained by computing the MVEE of a centrally symmetric point set. Therefore, we can focus on addressing model ($P$) and ($D_2$).

\subsection{Algorithm}
The classical method to address model ($P$) is the Newton type algorithm (see \cite{Sun2004}), while the classical method to address model ($D_2$) is the Frank-Wolfe (FW) type algorithm (see \cite{Khachiyan1996}, \cite{Kumar2005}, \cite{Ahipasaoglu2008}, \cite{Cong2012}, \cite{Cong2017}). Numerical experiments show that the FW type algorithm is more efficient for the case of large-scale or huge-scale data sets (especially for high-dimensional data sets) than the Newton type algorithm (\cite{Todd2016}). \cite{Atwood1973}
was the first to apply the FW type algorithm to compute the MVEE (in fact, Atwood was trying to address the D-optimal design problem in statistics, which is the same as the duality problem in MVEE. Therefore, Atwood can be considered the first to apply the FW algorithm to the MVEE problem). \cite{Khachiyan1996} was the first to obtain the computational complexity of the FW type algorithm on the MVEE problem. \cite{Kumar2005} extended the work of \cite{Khachiyan1996}, and proposed a new initialization scheme to improve the efficiency of the algorithm. Based on the work of Kumar and Yilidrim, \cite{Ahipasaoglu2008} proposed the ``dropping points'' technique and discovered the local linear convergence rate of the FW algorithm on the MVEE problem. Table 1 compares the computational complexity of different classical methods for solving the MVEE problem.

\begin{table}\label{tab1}
\centering
\begin{tabular}{|c|c|}
\hline
Algorithm & Complexity \\
\hline
Interior Point Algorithm & $O(m^{3.5} \ln (\frac{m}{\epsilon}))$ \\
\hline
FW-K  Algorithm & $O(mn^2([(1+\epsilon)^{\frac{2}{n+1}}-1]^{-1}+\ln n +\ln \ln m))$\\
\hline
FW-KY Algorithm & $O(mn^2([(1+\epsilon)^{\frac{2}{n+1}}-1]^{-1}+\ln n))$\\
\hline
\end{tabular}
\caption{Computation complexity of classical algorithms to compute the MVEE.}
\end{table}

In the following, the FW-K and WA algorithms are introduced as examples. Suppose the optimal solution pair is $(H^{*},u^{*})$, then the optimality condition yields:
\[
\begin{array}{l}
u_i^{*} \geq 0, i=1, \ldots, m,\\
e^{T} u^{*} = 1,\\
x_i^T H^{*} x_i \leq n, i=1, \ldots, m,\\
u_i^{*} (x_i^T H^{*} x_i - n) = 0, i=1, \ldots, m,\\
(X U^{*} X^T)^{-1} = H^{*}.\\
\end{array}
\]
Therefore, $u$ is the optimal point of model ($D_2$) if and only if $H(u) = (XUX^T)^{-1}$ is feasible for model ($P$). Based on this, \cite{Khachiyan1996} proposed the notion of the $\epsilon$-primal feasible solution.

\begin{definition}
Suppose $u$ is a feasible solution of ($D_2$), $u$ is said to be $\epsilon$-primal feasible if $H(u)=(XUX^T)^{-1}$ satisfies $x_i^T H(u) x_i \leq (1+\epsilon)n,i=1,\ldots,m$.
\end{definition}

In the FW-K algorithm, the objective function is linearized to generate subproblems, and the optimal solutions of the subproblems are used to generate the descent direction. Using this direction and the corresponding optimal step size, model ($D_2$) can be solved iteratively. If the current iteration point is $u_k$, the linearized subproblem of model ($D_2$) is:
\[
\begin{array}{cl}
\max_u & \kappa(u_k)^T u \\
st. & u \in \triangle_m \\
\end{array}
\]
where $\kappa(u_k) =\nabla g(u_k)$, and $\kappa(u_k)_i$ represents the $i$-th component of $\kappa(u_k)$. Suppose $j=\textrm{arg}\max_{i=1,\ldots,m} \kappa(u_k)_i$, and then $e_j$ is the optimal solution of the above subproblem. So if the optimal step size is $\lambda_k$, the new iteration point is $u_{k+1}=u_k + \lambda_k(e_j - u_k)$, or equivalently, $u_{k+1}$ can be seen as a convex combination of $u_k$ and $e_j$. As $\lambda_k$ is chosen to minimize the objective function $g(u)$, it follows that:
\[
\frac{d}{d\lambda}\ln \det (u_k+\lambda(e_j - u_k)) = 0 \Rightarrow \lambda_k = \frac{\kappa_j - n}{n(\kappa_j - 1)}.
\]
We describe the FW-K algorithm as follows:
\begin{algorithm}[H]
\caption{FW-K algorithm for model ($D_2$)}\label{1stAlgorithm}
\begin{algorithmic}[1]
\STATE {Input:$X=\{x_1, \ldots, x_m\}$, $\epsilon > 0$;}
\STATE {Initialization: $u_0=(\frac{1}{m}, \ldots, \frac{1}{m})$;}
\WHILE {not converged}
\STATE {$j$ $\leftarrow \arg\max_{i=1,\ldots, m} \kappa(u_k)^i$,$\lambda_k = \frac{\kappa_j - n}{n(\kappa_j - 1)}$, $u_{k+1} \leftarrow u_k + \lambda_k (e_j - u_k)$}
\STATE {Output: $u^{*}=u_k$}
\ENDWHILE
\end{algorithmic}
\end{algorithm}

In the FW-K algorithm, the initial point $u_0$ is initialized as $u_0=(\frac{1}{m}, \ldots, \frac{1}{m})$. \cite{Kumar2005} proposed a core set based initialization scheme, which is more effective than Khachiyan's initialization scheme. Furthermore, since the FW-K algorithm, as well as the FW-KY algorithm, yields only an $\epsilon$-primal feasible solution, which fails to satisfy the complementary slackness condition, \cite{Ahipasaoglu2008} proposed the notion of $\epsilon$-approximately optimal solution to address this deficiency (see also \cite{Todd2007}, \cite{Todd2016}):

\begin{definition}
If $u$ is an $\epsilon$-primal feasible solution satisfying $x_i^T H(u) x_i \geq (1-\epsilon)n,i=1,\ldots,m$ when $u_i > 0$, then $u$ is an $\epsilon$-approximately optimal solution.
\end{definition}

To improve the FW-K algorithm, \cite{Todd2007} and \cite{Ahipasaoglu2008} introduced the ``away-step'' strategy of the Frank-Wolfe algorithm, resulting in the Wolfe-Atwood (WA) algorithm.

\begin{algorithm}[H]
\caption{WA algorithm for model ($D_2$)}\label{2ndAlgorithm}
\begin{algorithmic}[1]
\STATE {Input:$X=\{x_1, \ldots, x_m\}$, $\epsilon > 0$;}
\STATE {Initialization: Using the Kumar - Yildirim initialize scheme to initialize $u_0$;;}
\WHILE {not converged}
\STATE {$j_{+} \leftarrow \arg \max_{i=1,\ldots, m} \kappa(u_k)_i$, $\kappa_{+} \leftarrow \kappa(u_k)_{j+}$; $j_{-} \leftarrow \arg \min_{i|u_i>0} \kappa(u_k)_i$, $\kappa_{-} \leftarrow \kappa(u_k)_{j-}$;$\epsilon_{+} \leftarrow \frac{\kappa_{+}}{n}-1, \epsilon_{-} \leftarrow 1- \frac{\kappa_{-}}{n}$, $\epsilon_k \leftarrow \max\{\epsilon_{+}, \epsilon_{-}\}$;}
\IF {$\epsilon_k = \epsilon_{+}$, then $\lambda_k = \frac{\kappa_j - n}{n(\kappa_j - 1)}$,}
\STATE {$u_{k+1} \leftarrow u_k + \lambda_k (e_{j+} - u_k)$},
\ELSE
\STATE {$\lambda_k = \min \{\frac{n - \kappa_j}{n(\kappa_j - 1)}, \frac{u_{j-}}{1 - u_{j-}}\}$, $u_{k+1} \leftarrow u_k + \lambda_k ( u_k- e_{j-})$;}
\ENDIF
\ENDWHILE
\STATE {Output: $u^{*}=u_k$}
\end{algorithmic}
\end{algorithm}

There are certain noteworthy points about the WA algorithm proposed by \cite{Todd2007}
\begin{enumerate}
\item In the $k$-th iteration, some of the $m$ points fall inside the ellipsoid, corresponding to $x_i^T (XUX^T)^{-1} x_i \leq n$, while the other points fall outside of the ellipsoid, corresponding to $x_i^T (XUX^T)^{-1} x_i > n$. The goal of the current iteration is to make each element of the gradient of the objective function (i.e., $x_i^T (XUX^T)^{-1} x_i$) as close to $n$ as possible until $(1-\epsilon) n \leq x_i^T (XUX^T)^{-1} x_i \leq (1+\epsilon)n$ holds for each $i$.

\item Choose an $i \in I$ and set $u^{\dagger} = u+\lambda (e_i - u)$, where $\{x_i | i \in I\}$ denotes the set of points outside the ellipsoid currently. The equality $x_i^T (XU^{\dagger}X^T)^{-1} x_i=n$, or equivalently,
\[
x_i^T [(1-\lambda)XUX^T + \lambda x_i x_i^T]^{-1} x_i =n,
\]
yields $\lambda = \frac{\kappa(u_k)_i - n}{n( \kappa(u_k)_i-1)}$, which is consistent with the precise step size $\lambda$ obtained in Khachiyan's algorithm.

This $\lambda$ can guarantee the selected point $x_i$ to fall on the boundary of the ellipsoid in the next iteration. Similarly, if the selected point $x_j$ is inside the ellipsoid, choosing $\lambda = \frac{n - \kappa(u_k)_j}{n( \kappa(u_k)_j -1)}$ as the step size also makes the point $x_j$ satisfy $x_j^T H(u^{\dagger}) x_j = n$ in the next iteration, i.e., fall on the boundary of the ellipsoid in the subsequent iteration. In this way, each component of the gradient of the objective function asymptotically approaches to $n$.

\item In order to compute the initial gradient $\kappa(u_0)$, we can resort to the scaled Cholesky factorization with $\frac{1}{3}n^3$ arithmetic operations to obtain $XUX^T = \phi L L^T$. On one hand it is convenient to compute $(XUX^T)^{-1}$ via the Cholesky factorization. On the other hand, the Cholesky factorization of $(XU^{\dagger}X^T)$ corresponding to the next iteration point $u^{\dagger}$ can be obtained through Cholesky rank-one modification formula: $\frac{\phi}{1-\lambda} (XU^{\dagger}X^T) = LL^T + \frac{\lambda}{1-\lambda} \phi x_j x_j^T$. In addition, it is also convenient to update the inverse matrix $(XU^{\dagger}X^T)^{-1}$ via the rank-one modification formula. Specifically:
\[
\begin{array}{cl}
(XU^{\dagger}X^T)^{-1} & = ((1 - \lambda)XUX^T + \lambda x_j x_j^T)^{-1} \\
\  & = \frac{1}{1-\lambda} [(XUX^T)^{-1} + \frac{\lambda}{1+\lambda(\kappa(u)_j - 1)} (XUX^T)^{-1} x_j x_j^T (XUX^T)^{-1}] \\
\end{array}
\]
Assume we have performed the scaled Cholesky decomposition on $XUX^T$, then it is easy to obtain $\hat{x}_j = (XUX^T)^{-1} x_j$ with $2n^2$ arithmetic operations, so that the gradient in the next iteration $\kappa(u^{\dagger})_i = \frac{1}{1-\lambda} [\kappa(u)_i + \frac{\lambda (x_i^T \hat{x}_j)^2}{1+\lambda (\kappa(u)_j - 1)}]$ can be computed explicitly.
\end{enumerate}

\section{Coordinate Descent Algorithm}
\subsection{``Algorithmic'' Coordinate-wise Smoothness}
In this paper we propose a gradient descent type algorithm (i.e., the Coordinate Descent algorithm) to address the MVEE problem. We mention that the classical algorithms on MVEE are Frank - Wolfe type (LOO) algorithms whose iteration direction is not the gradient descent direction. Although the algorithm proposed in this paper has many internal connection with the Frank - Wolfe type algorithm, its iteration direction is exactly the gradient descent direction. The comparison of our algorithm and the Frank - Wolfe type algorithm (exemplified as the WA algorithm) is elaborated in the 5th section.

First, several relevant definitions needed in the subsequent text, such as smoothness, strong convexity, and coordinate-wise smoothness are given as follows.

\begin{definition}
A continuously differentiable function $f(x)$ is said to be $L$-smooth on $\mathbb{R}^n$ if there exists a constant $L > 0 $ such that for any $x,y \in \mathbb{R}^n$ we have $\| \nabla f(x) - \nabla f(y) \| \leq L \|x-y\|$. Here $L$ is called the smoothness parameter.
\end{definition}

\begin{definition}
A continuously differentiable function $f(x)$ is said to be $\mu$-strongly convex on $\mathbb{R}^n$ if there exists a constant $\mu > 0 $ such that for any $x,y \in \mathbb{R}^n$ we have $\| \nabla f(x) - \nabla f(y) \| \geq \mu \|x-y\|$. Here $\mu$ is called the convexity parameter.
\end{definition}

\begin{definition}
A continuously differentiable function $f(x)$ is said to be coordinate-wise smooth with smoothness parameter $L_i$ if there exists $L_i>0$ such that for any $x \in \mathbb{R}^n$, $\theta \in \mathbb{R}$, we have $|\nabla_i f(x+\theta e_i) - \nabla_i f(x)| \leq L_i |\theta|$, $i=1,\ldots,n$.
\end{definition}

As we know, the objective function of model($D_1$) or ($D_2$) is neither smooth nor strongly convex, as a result we here introduce a new kind of coordinate-wise smoothness called the ``algorithmic'' coordinate-wise smooth, and use this new ``smoothness'' to analyze convergence properties.

\begin{definition}
A continuously differentiable function $f(x)$ is said to be ``algorithmic'' coordinate-wise smooth w.r.t the algorithm $\mathcal{A}$, if there exists a coordinate descent algorithm $\mathcal{A}$, such that the two successive points of algorithm $\mathcal{A}$, namely $u_k$ and $u_{k+1}$,(suppose all the elements of $u_k$ and $u_{k+1}$ but the $i_k$th element are equal.) satisfies $|\nabla_{i_k} f(x+\theta e_{i_k}) - \nabla_{i_k} f(x)| \leq L_{i_k} |\theta|$ for some positive constant $L_{i_k}$.
\end{definition}

In the following, we discuss the properties of the objective function of model ($D_2$).

\begin{lemma}\quad
\label{lemma1}
Let $u$ be a feasible solution of model ($D_2$), $\kappa(u) = \nabla_{u} g(u)$, and $\kappa(u)_i$ is the $i$th coordinate of $\kappa(u)$, then:

$\forall \theta_i \geq 0$, $i=1,\ldots,m$, we have $ |\kappa(u+\theta_i e_i)_i - \kappa(u)_i| \leq (\kappa(u)_i)^2 |\theta_i|$. $\forall \theta_i$ satisfying $\frac{\kappa(u)_i - n}{\kappa(u)_i n} \leq \theta_i \leq 0$ with $\kappa(u)_i \leq n$, we have $|\kappa(u+\theta_i e_i)_i - \kappa(u)_i| \leq n \kappa(u)_i |\theta_i|$.
\end{lemma}
\begin{proof}\quad
For $\theta_i \geq 0$, by the definition of $\nabla g(u)$, and the Sherrman-Morrison-Woodbury formula, we have
\[
\begin{array}{cl}
|\kappa(u+\theta_i e_i)_i - \kappa(u)_i| & = |x_i^T (XUX^T + \theta_i x_i x_i^T)^{-1} x_i - x_i^T (XUX^T)^{-1} x_i |\\
\ & = |x_i^T (\frac{-\theta_i (XUX^T)^{-1} x_i x_i^T (XUX^T)^{-1}}{1+\theta_i \kappa(u)_i}) x_i| \\
\ & = \frac{\theta_i (\kappa(u)_i)^2}{1+\theta_i \kappa(u)_i} \\
\ & \leq (\kappa(u)_i)^2 |\theta_i|. \\
\end{array}
\]
Hence, for $\frac{\kappa(u)_i - n}{\kappa(u)_i n} \leq \theta_i \leq 0$ with $\kappa(u)_i \leq n$, we have $\frac{\kappa(u)_i^2}{1+\kappa(u)_i \theta_i} \leq \kappa(u)_i n$, and it follows that
\[
\begin{array}{cl}
|\kappa(u+\theta_i e_i)_i - \kappa(u)_i| & = |x_i^T (XUX^T + \theta_i x_i x_i^T)^{-1} x_i - x_i^T (XUX^T)^{-1} x_i|\\
\ & = |x_i^T (\frac{-\theta_i (XUX^T)^{-1}x_i x_i^T (XUX^T)^{-1}}{1+\theta_i \kappa(u)_i})x_i|\\
\ & = -\frac{\theta_i \kappa(u)_i^2}{1+\theta_i \kappa(u)_i}\\
\ & \leq n \kappa(u)_i |\theta_i|.\\
\end{array}
\]
\end{proof}

Lemma 1 presents a kind of coordinate-wise property of the objective function $g(u)$, which is weaker than the coordinate-wise smoothness, since the Lipschitz constant of smoothness $L_i, i=1,\ldots,m$ is relevant to $\kappa(u)_i$. That is $L_i = (\kappa(u)_i)^2$ for $\theta \geq 0$ and $L_i=n(\kappa(u)_i)$ for $ \frac{\kappa(u)_i - n}{\kappa(u)_i n} \leq \theta <0$. We note that if the iteration points generated by some algorithm satisfies the condition in Lemma 1, then the objective function $g(u)$ of ($D_1$) is ``algorithmic'' coordinate-wise smooth, and so does the objective function $h(u)$ of ($D_2$). We emphasize here that the objective function $g(u)$ for ($D_1$) or $h(u)$ for ($D_2$) is exactly not coordinate-wise smooth, since the inequality in lemma 1 may not guaranteed in each feasible point $u$. The ``algorithmic'' coordinate-wise smoothness is rather weak, but it is enough to guarantee the convergence of our proposed algorithm.

According to \cite{Nesterov2011}, the coordinate-wise smoothness can be extended to the smoothness. Since if $g$ is twice differentiable, the coordinate smoothness is equivalent to diagonal elements of the Hessian are bounded above, so that the maximum eigenvalue of the Hessian is bounded above by its trace, and the smoothness is obtained. In comparison, we note that the strong convexity cannot be obtained by coordinate-wise strongly convex. Since the coordinate strong convexity only means the diagonal elements are bounded below which cannot guarantee the minimum eigenvalue of the Hessian is greater than 0. Lemma 1 also reveals the possibility of introducing the FOO algorithm to compute the MVEE problem. However, this possibility does not mean that the FOO algorithm can be applied directly to solve model ($D_2$).

The reason is that we cannot always guarantee the positive definiteness of the matrix $XUX^T$. Note that $XUX^T$ is at least a positive semi-definite matrix, and $\forall 0 \neq y \in \mathbb{R}^n$ such that $y^T (XUX^T) y = \sum\limits_{\{i | u_i>0\}} u_i \| y^T x_i\|^2=0$, we have $y \perp x_i, \forall i \in I=\{i | u_i >0 \}$. It follows that if $aff\{x_i|i \in I\}=\mathbb{R}^n$, then such $y$ dose not exist, and $XUX^T$ is always positive definite.

As for the FW type algorithms, no matter we choose the Khachiyan initialization scheme ($u_0=\{\frac{1}{m},\ldots, \frac{1}{m}\}$), or the Kumar-Yildirim initialization scheme (non-zero element in $u_0$ represents $n$'s linearly independent directions), the initialization of $u_0$ satisfies the positive definiteness of $XU_0X^T$. Moreover, since the new iteration point $u_{k+1}$ is the convex combination of the current iteration point $u_k$ and $e_j$, where $j$ is the selected coordinate, the number of nonzero elements in $u_{k+1}$ will not decrease with respect to $u_{k}$, $\forall k$. As a result, $XU_kX^T$ is positive definite, $\forall k$. However, for the FOO algorithm, e.g., the gradient  projection algorithm, the projection to the probability simplex is equivalent to finding one $\lambda$ such that $\sum\limits_{j=1}^m \max\{u_j - \lambda, 0\} =1$(see \cite{Boyd2004}), and each element of the new iteration point $u_{k+1}$ is $u_{k+1}^j = \max\{u_{k}^j -\lambda,0\}$, $j=1,\ldots,m$. It follows that the number of nonzero elements in $u_{k+1}$ might be less than $n$. As a result, simply using the gradient projection method does not guarantee the positive definiteness of $XU_kX^T$.

\subsection{Coordinate Descent Algorithm with ``Constant'' Stepsize}
In this section, we  propose the Coordinate Descent (CD) algorithm to solve model ($D_1$) instead of model ($D_2$) so as to compute the MVEE. Denoting $h(u) = g(u) + n (1 - e^T u)$ as the objective function of model ($D_1$), the CD algorithm is:

\begin{algorithm}[H]
\caption{Coordinate Descent algorithm for model ($D_1$)}\label{3rdAlgorithm}
\begin{algorithmic}[1]
\STATE {Input:$X=\{x_1, \ldots, x_m\}$, $\epsilon > 0$}
\STATE {Initialization: Using Khachiyan's or Kumar - Yildirim's initialize scheme to initialize $u_0$}
\WHILE {not converged}
\STATE {Choose $j=\arg \max_{i=1,\ldots,m} {|\nabla h(u_k)_i|}$}
\STATE {$u_{k+1}=P_{u \geq 0}(u_k - \frac{1}{L_j}\nabla h(u_k)_j )$}
\ENDWHILE
\STATE {Output: $u^{*}=u_k$}
\end{algorithmic}
\end{algorithm}

There are several points needed to be explained with regard to the CD algorithm.
\begin{enumerate}
\item The CD algorithm above is actually the gradient type coordinate descent algorithm with the Gauss - Southwell's Rule to choose the coordinate axis.
  Notice that the CD algorithm aims to compute model ($D_1$) whose objective function is $h(u)$ with the gradient $\nabla h(u) = ne - \kappa(u)$ satisfying $| \nabla_i h(u+\theta e_i) - \nabla_i h(u) |=|\kappa(u+\theta e_i)_i -\kappa(u)_i| \leq L_i |\theta|$. According to lemma 1, in the first case, if the $j$th coordinate-axis with $n \leq \kappa(u)_j$ is chosen, the Lipchitiz constant $L_j$ is set to be $(\kappa(u)_j)^2$, and the iteration can be written as (suppose $u$ and $u^{\dagger}$ denote two successive iteration points):
  \[
  \begin{array}{cl}
  u^{\dagger} & = P_{u\geq 0}(u-\frac{1}{L_j} \nabla h(u)_je_j)\\
  \ & = u - \frac{1}{(\kappa(u)_j)^2} (n - \kappa(u)_j)e_j\\
  \ & = u + \frac{\kappa(u)_j-n}{(\kappa(u)_j)^2}e_j\\
  \end{array}
  \]
 In the second case, if the $j$th coordinate-axis with $n \geq \kappa(u)_j$ is chosen, the Lipchitz constant $L_j$ is set to be $n\kappa(u)_j$, and the iteration can be written as:
  \[
  \begin{array}{cl}
  u^{\dagger} & = P_{u\geq 0}(u-\frac{1}{L_j} \nabla h(u)_je_j)\\
  \ & = P_{u\geq 0}({u - \frac{1}{(n\kappa(u)_j)} (n - \kappa(u)_j)e_j})\\
  \ & = \left\{\begin{array}{cl}
  u + \frac{\kappa(u)_j-n}{(n\kappa(u)_j)}e_j & \textrm{if} \ u_j \geq \frac{n-\kappa(u)_j}{(n\kappa(u)_j)}\\
  0 & \textrm{if} \  u_j < \frac{n-\kappa(u)_j}{(n\kappa(u)_j)}\\
  \end{array}\right.\\
  \end{array}
  \]
  It is not difficult to see that the first case above is equal to the ``add'' and ``increase'' steps of the WA algorithm, while the second case corresponds to the ``minus'' and ``drop'' steps of the WA algorithm. Therefore, we can similarly classify the iteration steps of the CD algorithm into the ``add'', ``increase'', ``decrease'' and ``drop'' types. Moreover, Notice that $ n -\kappa(u)_i$ is exactly the $i$th element of the gradient of the objective function of model ($D_1$). Also recall that the coordinate selection rule in the WA algorithm is that: when $\kappa_{+} - n \geq n - \kappa_{-}$, the $j+$th coordinate axis is selected, otherwise the $j-$th coordinate axis is selected. {\bf Therefore, the iterative direction rule of the WA algorithm (the ``ordinary step'' and the ``away step'' of the Frank - Wolfe algorithm) is $\max_{i=1,\ldots,m} \{|\nabla_i h(u)|\}$, which is exactly the same as the ``Gauss - Southwell's Rule''. It is noted that if we resort to problem (D2) instead of (D1), the ``Gauss - Southwell's Rule'' is hide}. Moreover, we see clearly that the objective function $h(u)$ is ``algorithmic coordinate-wise smooth'' w.r.t the CD algorithm.
\item In the ``add'' and ``increase'' iteration step, since $\kappa(u^{\dagger})_{j+} - \kappa(u)_{j+} = \frac{(n-\kappa_{+})\kappa_{+}}{2\kappa_{+}-n} <0$, we can only guarantee that $x_{j+}^T (XU^{\dagger}X^T) x_{j+} \leq  x_{j+}^T (XUX^T) x_{j+}$, while we can not guarantee that the value of $x_{j+}^T (XU^{\dagger}X^T) x_{j+}$ can be decreased to $n$. However, if the step size is set as $\lambda =\frac{n-\kappa_{+}}{n \kappa_{+}}$, $x_{j+}^T (XU^{\dagger}X^T) x_{j+}$ can be reduced to $n$. Unfortunately, since $\kappa_{+} > n$, ${n \kappa_{+}}$ is not the smoothness parameter of the $j+$th coordinate.
\end{enumerate}
\subsection{New Explanation of the Frank - Wolfe Type algorithms}
Based on the analysis above, we can also explain the FW-K algorithm from a new perspective. \cite{Kumar2005} explained the FW-K algorithm as follows: each iteration step of the FW-K algorithm is improving toward the direction of the point with the largest ellipsoidal norm ($\|\cdot\|_H$, where $H=XU_kX^T$). \cite{Sun2004} treated the ellipsoidal center as the mean value of the point set $X$, and the $\|\cdot\|_H$ distance of a point in the ellipsoid to the ellipsoidal center as the variance of the point. Thus, each iteration step of the FW-K algorithm is improving towards the direction of the point with the largest variance. Lemma1 shows that the objective function $g$ is $L_i=\kappa(u)_i^2$  coordinate-wise smooth when $\theta_i \geq 0$, $i=1,\ldots,m$. {\bf Notice that in each iteration, the FW-K algorithm chooses the coordinate axis $i$ with the largest value of $\kappa(u)^i$, which corresponds to the coordinate axis with the worst smoothness. This coordinate selection scheme is similar to that proposed by \cite{Nesterov2011}, that is, the worse smoothness the coordinate axis is, the more likely it is to be selected in the next iteration. Numerous numerical tests have shown that the ``Gauss-Southwell's Rule'' is faster than other coordinate axis selection rules \cite{Nutini2015}, and it provides another explanation of why the WA algorithm is much faster than the FW-K algorithm on the MVEE problem.}
\subsection{Comparison of the CD and WA algorithm}
We compare the WA and CD algorithm in the following table. In table 2, the CD and WA algorithm both use the ``Gauss - Southwell's Rule'' to choose the coordinate axis. In \cite{Nutini2015}, the ``Gauss - Southwell's Rule'' is an efficient rule of choosing coordinate axis for coordinate descent methods. The CD algorithm chooses the $j$th element of the negative gradient as the iterative direction, while the WA algorithm chooses $u - e_{j+}$ or $e_{j-} - u$ as the iterative direction. Both of these two directions are descent directions for the objective function. The WA uses the accurate stepsize at each iteration, because the accurate stepsize can be computed in closed from. The CD algorithm uses the ``constant'' stepsize which is simple to compute and can lead to sufficient decrement of the objective function at each iteration. It is noted that since the Lipschitz constant $L_i$ is changing w.r.t each iteration, this kind of ``constant stepsize'' is actually ``dynamic stepsize''. Finally, we note that in the ``black-box'' model, the WA and CD algorithm are two different types of algorithms.

\begin{table}\label{tab2}
\centering
{\footnotesize \begin{tabular}{|c|c|c|}
\hline
Algorithm & WA & CD\\
\hline
Rules to Choose Coordinate Axis & Gauss - Southwell's Rule & Gauss - Southwell's Rule\\
\hline
Direction & $u - e_j$ & $-\kappa(u)_j e_j$\\
\hline
Stepsize & $\arg \min_{\lambda} \{- \ln \det (XU(\lambda)X^T)\}$ & $\frac{1}{L_j}$\\
\hline
Oracle & LOO & FOO\\
\hline
\end{tabular}}
\caption{Comparison of the WA algorithm and CD algorithm.}
\end{table}

\section{Convergence Analysis of the algorithm}
\subsection{Convergence Properties of the CD algorithm}
In this section we discuss the convergence properties of the CD algorithm. We first prove that the CD algorithm is sublinearly convergent, and then we discuss on the ``globally (locally)'' linear convergence properties of the WA algorithm mentioned in \cite{Todd2016}.
\begin{lemma}\quad
\label{lemma2}
Let $u^{\dagger}$ be the next iteration point of $u$ generated by the CD algorithm, then:
\[
h(u) - h(u^{\dagger}) \geq \left\{ \begin{array}{ll}
\frac{(n-\kappa_{+})^2}{2 \kappa_{+}^2}, & \textrm{if} \ j+ \ \textrm{is chosen} \\
\frac{(n-\kappa_{-})^2}{2 n \kappa_{-}}, & \textrm{if} \ j- \ \textrm{is chosen and} \; \frac{n - \kappa_{-}}{\kappa_{-} n} < u_{j-}\\
0, & \textrm{otherwise}. \\
\end{array}
\right.
\]
\end{lemma}
\begin{proof}\quad
First we notice that with respect to the two successive points generated by the CD algorithm, the objective function $h(u)$ is ``algorithmic'' coordinate-wise smooth w.r.t the CD algorithm. More specifically, it is $\kappa(u)_{j+}^2$ coordinate-wise smooth when coordinate $j+$ is selected, and $n \kappa(u)_{j-}$ coordinate-wise smooth when coordinate $j-$ is selected. By the above relation and the fact that $\nabla_i h(u) = n - \kappa(u)_i$, we have $u^{\dagger} = u - \frac{n-\kappa_{+}}{\kappa_{+}^2} e_{j+}$ if the current iteration chooses the $j+$th coordinate. Using the smoothness inequality (\cite{Nesterov2004}), we obtain
\[
\begin{array}{cl}
h(u^{\dagger}) & \leq h(u) -(n-\kappa_{+}) \frac{n-\kappa_{+}}{\kappa_{+}^2} + \frac{\kappa_{+}^2}{2} \frac{(n-\kappa_{+})^2}{\kappa_{+}^4},\\
\ & = h(u) - \frac{(n-\kappa_{+})^2}{2(\kappa_{+})^2}.\\
\end{array}
\]
If the current iteration chooses the $j-$th coordinate, we have $u^{\dagger} = P_{\{u| u \geq 0\}}(u - \frac{n - \kappa_{-}}{\kappa_{-} n} e_{j-})$, and from the famous projection theorem (\cite{Bertsekas2015}), we have
\[
(u - \frac{n-\kappa_{-}}{\kappa_{-}n} e_{j-} - u^{\dagger})^T (u-u^{\dagger}) \leq 0,
\]
i.e., $(n-\kappa_{-})e_{j-}^T (u^{\dagger} - u) \leq -(\kappa_{-}  n) \|u-u^{\dagger}\|^2$. It follows that
\[
\begin{array}{cl}
h(u^{\dagger}) & \leq h(u) - (n \kappa_{-}) \|u-u^{\dagger}\|^2 + \frac{(n\kappa_{-})}{2} \|u-u^{\dagger}\|^2 \\
\ & = h(u) - \frac{n \kappa_{-}}{2} \|u-u^{\dagger}\|^2. \\
\end{array}
\]
Here, in case $\frac{n - \kappa_{-}}{\kappa_{-} n} < u_{j-}$, we have $\|u-u^{\dagger}\|^2 = \frac{1}{(n \kappa_{-})^2} (n-\kappa_{-})^2$, and therefore:
\[
\begin{array}{cl}
h(u^{\dagger}) & \leq h(u) - \frac{n \kappa_{-}}{2} \frac{1}{(n \kappa_{-})^2} (n-\kappa_{-})^2 \\
\ & \leq h(u) - \frac{(n-\kappa_{-})^2}{2(n \kappa_{-})}. \\
\end{array}
\]
In case $\frac{n - \kappa_{-}}{\kappa_{-} n} \geq u_{j-}$, since the value of $\|u-u^{\dagger}\|^2$ might be arbitrarily small, it can only be guaranteed that:
\[
h(u^{\dagger}) \leq h(u).
\]
\end{proof}

Lemma \ref{lemma2} shows that the CD algorithm generates a sequence of relaxation points (\cite{Nesterov2004}) which provides a good condition for the convergence analysis. Indeed, let $u^{\dagger}=u+\lambda e_i$, then $h(u)-h(u^{\dagger}) = \ln (1+\lambda \kappa(u)_i) - n\lambda$.
Therefore, if we choose the $j+$th coordinate in the current iteration, i.e., $i=j+$, $\lambda=\frac{\kappa_{+}-n}{\kappa_{+}^2}$, then:
\[
h(u)-h(u^{\dagger}) = \ln(2-\frac{n}{\kappa_{+}}) - \frac{n(\kappa_{+} - n)}{\kappa_{+}^2},
\]
which can be proved to be greater than $\frac{(n-\kappa_{+})^2}{2\kappa_{+}^2}$.

If we choose the $j-$th coordinate in the current iteration and $\frac{\kappa_{-}-n}{n\kappa_{-}} \geq - u_{j-}$, i.e., $i=j-$, $\lambda=\frac{\kappa_{-}-n}{n\kappa_{-}}$, then:
\[
h(u)-h(u^{\dagger}) = \ln(\frac{\kappa_{-}}{n}) + \frac{n}{\kappa_{-}} - 1,
\]
which can be proved to be greater than $\frac{(n-\kappa_{-})^2}{2n\kappa_{-}}$.

If we choose the $(j-)$th coordinate in the current iteration and $\frac{\kappa_{-}-n}{n\kappa_{-}} < - u_{j-}$, i.e., $i=j-$, $\lambda=-u_{j-}$, then :
\[
h(u)-h(u^{\dagger}) = n u_{j-} + \ln (1-u_{j-} \kappa_{-}),
\]
which can be proved to be greater than $0$(see Appendix).

Next we shall show the coercivity of $h(u)$.
\begin{lemma}\quad\label{lemma3}
The objective function $h(u)$ of model ($D_1$) is coercive over the set $\{ u | u \geq 0\}$.
\end{lemma}
\begin{proof}\quad
By the definition of coercivity, to show that $h(u) \rightarrow \infty$ when $u \geq 0$ and $\|u\| \rightarrow \infty$ is enough. It is easy to see that $\|u\| \rightarrow \infty$ if and only if at least one coordinate $u_i$ of $u$ satisfies that $u_i \rightarrow \infty$. As $h(u^{'}) = - \ln \det (XUX^T) - \ln (1+ t \kappa(u)_i) + n(e^T u - 1) + nt$, where $u^{'}=u+t\cdot e_i$, we know $h(u^{'}) \rightarrow \infty$ when $t \rightarrow \infty$. This completes the proof.
\end{proof}

From Lemma 3, we know that although the feasible domain of the model ($D_1$) is unbounded, we can restrict the attention into a compact set $dom(h) \cap \{u | u \geq 0\}$.

Combining the Lemmas 2 and 3, we can prove the convergence of the sequence generated by the CD algorithm.

\begin{theorem}\quad
The sequence $\{u_k\}_{k=1}^{\infty}$ generated by the CD algorithm converges to the optimal solution of problem $(D1)$
\end{theorem}
\begin{proof}\quad
By Lemma 2, if we only consider the add-, plus- and minus-iterations, we have:
\[
h(u) - h(u^{\dagger}) \geq \left\{
\begin{array}{l}
\frac{(n-\kappa_{+})^2}{2\kappa_{+}^2} = (\frac{\kappa_{+}-n}{\kappa_{+}^2})^2 \frac{(\kappa_{+})^2}{2} = \frac{L_{j+}(u)}{2}\|u^{\dagger}-u\|^2,\\
\frac{(n-\kappa_{-})^2}{2n\kappa_{-}} = (\frac{n-\kappa_{-}}{n\kappa_{-}})^2 \frac{n\kappa_{-}}{2} = \frac{L_{j-}(u)}{2}\|u^{\dagger}-u\|^2.\\
\end{array}
\right.
\]
As $\kappa(u)_i = x_i^T (XUX^T)^{-1}x_i$, we just need to show
\[
L_{j+}(u)=(\kappa_{+})^2=(x_{j+}^T(XUX^T)^{-1}x_{j+})^2 \geq L,
\]
and
\[
L_{j-}(u)=n\kappa_{-}=n(x_{j-}^T(XUX^T)^{-1}x_{j-}) \geq L,
\]
for some constant L. This reduces to show $\|(XUX^T)\|$ is upper bounded. As
\[
\|XUX^T\| =\max_{\|y\|=1}\|U^{\frac{1}{2}}X^T y\|^2 \leq (\|X^T\| \max_{i=1,\ldots,m}\sqrt{u_i})^2 =\|X^T\|_2^2 \max_{i=1,\ldots,m}u_i,
\]
therefore
\[
\begin{array}{ll}
x_i^T(XUX^T)^{-1}x_i & \geq \frac{\|x_i\|^2}{\|X^T\|^2 \max_{i=1,\ldots,m}u_i}, \\
\ & = \frac{\|x_i\|^2}{\|X^T\|^2 u_{\max}},\\
\ & \geq \frac{1}{\|X^T\|^2 u_{\max}},\\
\end{array}
\]
where $u_{\max}=\max_{i=1,\ldots,m}u_i$, and the last inequality is due to $\|x_i\|^2 \geq 1$ after shifting.

Denoting $L=\min\{\frac{n}{\|X^T\|^2 u_{\max}},\frac{1}{\|X^T\|^4 u_{\max}^2}\}$, we have $L_{j+} \geq L$, $L_{j-} \geq L$. Note from Lemma 3 we have $u_{\max}$ cannot be increased to infinity, and so that this L is exist.

Therefore we obtain that:
\[
h(u)-h(u^{\dagger}) \geq \frac{L}{2} \|u-u^{\dagger}\|^2,
\]
adding the above inequality from $k=0$ to $k=N$, we have:
\[
h(u_0) - h(u_N) \geq \frac{L}{2} \sum\limits_{i=1}^N \|u_{i+1} - u_i\|^2,
\]
and thus
\[
\sum\limits_{i=1}^N \|u_{i+1}-u_i\|^2 \leq \frac{2(h(u_0)-h(u_N))}{L},
\]
and therefore $\|u_{i+1}-u_i\| \rightarrow 0$, and $\|u_{i+1}-u_i\|^2=o(\frac{1}{N})$, and $u_{i+1}-u_i=o(\frac{1}{\sqrt{N}})$.

Moreover, from the first order optimality condition, since
\[
u^{\dagger}=\arg\min_{w\geq 0}\{\frac{1}{2}\|w-u+\frac{1}{L_j}\nabla h(u)_je_j\|^2\},
\]
so $\forall w \geq 0$, we have $(u^{\dagger}-u+\frac{1}{L_j}\nabla h(u)_j e_j)^T(w-u^{\dagger}) \geq 0$. Let $u^{*}$ be a limit point of the sequence $\{u_k\}$, there shall exist a subsequence $\{u_{k_i}\}$ converging to $u^{*}$, and
\[
(u_{k_i+1}-u_{k_i}+\frac{1}{L_j}\nabla h(u_{k_i})_j e_j)^T(w-u_{k_i+1}) \geq 0, \forall w \geq 0,
\]
therefore we have $\nabla h(u^{*})_j e_j^T (w-u^{*}) \geq 0$, $\forall w \geq 0$. In this sense, we have $\nabla h(u^{*})_j=0$ for $u_j^{*}>0$.

Moreover we have $\|\nabla h(u_k)_j e_j\| \leq \frac{\|u_{k+1}-u_k\|}{L_j} = o(\frac{1}{\sqrt{k}})$.
Therefore when only consider the ``add'', ``plus'' and ``minus'' iteration, the sequence converges and the number of iterations is $o(\frac{1}{\epsilon^2})$.

As for the drop-iterations, they can be paired with a previous iteration where either $u_{k_{j+}}$ was increased from zero, or $u_{k_{j-}}$ is decreased to zero for the first time which was positive at the initial iteration (\cite{Ahipasaoglu2008} and \cite{Todd2016}). From this we know that the number of the drop-iterations can be at most the sum of $o(\frac{1}{\epsilon^2})$ and the number of nonzero coordinates of $u_0$, which is $2n$ under Kumar-Yildirim's initialization scheme, and $m$ under Khachiyan's initialization scheme.

Then it can be concluded that the sequence generated by the CD algorithm is convergent to the optimal solution of problem $(D1)$.
\end{proof}

\begin{proposition}\quad
The optimal solution $u^{*}$ of problem (D1) and (D2) is unique.
\end{proposition}
\begin{proof}\quad
Denote $C_u=\{u|u\geq0\}$, and the set of optimal solutions as $X^{*}$. Since $h(u)$ and $C_u$ are closed and convex, then so is $X^{*}$ (see \cite{LuoTseng1992}). Assume $u_1, u_2 \in X^{*}$, then $\frac{u_1+u_2}{2} \in X^{*}$ follows as $X^{*}$ is convex. Therefore
\[
h(\frac{u_1+u_2}{2})=h(u^{*})=\frac{h(u_1)+h(u_2)}{2}.
\]
As the function $f(H)=-\ln \det H$ is strictly convex, we have $XU_1X^T=XU_2X^T$, and thus $u_1=u_2$.
\end{proof}
{\bf Theorem 1 and Proposition 1 show that the sequence generated by the CD algorithm converges to the unique optimal solution of problem (D1) $u^{*}$.}
\subsection{Global Convergence Rate of the CD algorithm}
Enlightened by the previous work of great significance in \cite{Khachiyan1996}, \cite{Ahipasaoglu2008} and \cite{Todd2016}, we can use lemma 2 and 3 to prove the globally sublinear convergence rate and the locally linear convergence rate of the CD algorithm.

We consider the globally convergence results. Suppose that the initial point $u_0$ is $\epsilon_0$ primal feasible, then $\epsilon_0 \leq m-1$ for the Khachiyan's initialization scheme (KIS), and $\epsilon_0 \leq n^5-1$ for the Kumar - Yildirim's initialization scheme (KYIS) (see Lemma 3 in \cite{Khachiyan1996}and Corollary 3.4 in \cite{Todd2016}). Denote $c=\max\{m-1, n^5-1\}$. Moreover we note that when $u$ is a $\delta$ primal feasible solution with $\delta >1$, only ``add-'' and ``increase-'' iteration activates, and thus $\delta$ is decreasing monotonically. Therefore we could apply the idea of \cite{Khachiyan1996} to divide the iterations into two stages: $1< \delta < c$, and $\delta \leq 1$.

For the first stage, suppose for each $u_k$, we denote $H(u_k)=(XU_kX^T)^{-1}$, and $f(H(u_k))=-\ln \det H(u_k)$. It is obviously that $f(H(u^{*}))$ is the optimal value of the primal problem (P) when $u^{*}$ is the optimal solution of the dual problem $(D1)$. Based on the analysis subsequent to lemma 2, we could bound the decrement of $f(H(u))$ of two successive points $u$ and $u^{\dagger}$, when $u$ is a $\delta$ primal feasible solution with $1<\delta\leq c$. Since
\begin{equation}
\begin{array}{ll}
f(H(u)) - f(H(u^{\dagger})) & \geq \ln (2-\frac{n}{\kappa(u)_{+}}),\\
\ & = \ln (2-\frac{1}{1+\delta}),\\
\ & \geq c_1 \ln (1+\delta),\\
\ & \geq \frac{c_1}{n} (f(H(u)))-f(H(u^{*})),\\
\end{array}
\end{equation}
where $c_1 = \frac{\ln(2m-1)}{\ln m}-1$ for the KIS, and $c_1=\frac{\ln(1+2n^5)}{1+n^5}-1$ for the KYIS. The last inequality is due to $f(H(u))-f(H(u^{*}))=g(u)-g(u^{*})\leq n \ln(1+\delta)$ (see Lemma 3 in \cite{Khachiyan1996}). Therefore we obtain that:
\[
\begin{array}{ll}
f(H(u_k))-f(H(u^{*})) & \leq (1-\frac{c_1}{n})^k (f(H(u_0))-f(H(u^{*}))),\\
\ & \leq
\left\{
\begin{array}{ll}
e^{-\frac{c_1 k}{n}} n \ln m, &\  \textrm{if KIS applies}, \\
e^{-\frac{c_1 k}{n}} 5 n \ln n, &\  \textrm{if KYIS applies}. \\
\end{array}
\right.\\
\end{array}
\]
Moreover, suppose $u_k$ is $\epsilon_k$ primal feasible with $1< \epsilon_k \leq c$
\begin{equation}
\begin{array}{ll}
f(H(u_k))-f(H(u^{*})) & \geq f(H(u_k))-f(H(u_{k+1})),\\
\ & \geq c_1 \ln (1+\epsilon_k),\\
\ & \geq c_2,\\
\end{array}
\end{equation}
where $c_2=c_1 \ln2$. Therefore by (1) and (2), we conclude that when KIS is applied to the CD algorithm, we obtain a 1 primal feasible solution within $\frac{n}{c_1} (\ln\frac{n \ln m}{c_2})$ steps; when KYIS is applied to the CD algorithm, we obtain a 1 primal feasible solution within $\frac{n}{c_1} (\ln\frac{5 n \ln n}{c_2})$ steps.

For the second stage, we use the decrement of $h(u)$ to bound the iteration steps. Suppose $u$ is a $\delta$ approximately optimal solution with $0< \delta \leq 1$, then we have:
\begin{equation}
\begin{array}{ll}
h(u) - h(u^{*}) & \leq g(u) -g(u^{*}) + n(e^T u-1),\\
\ & \leq n \ln (1+\delta) + \sum\limits_{\{i|u_i >0\}} u_i (n-\kappa(u)_i),\\
\ & \leq n \ln (1+\delta) + mn\delta u_{\max},\\
\ & \leq n(1+m u_{\max})\delta.\\
\end{array}
\end{equation}
The third inequality is due to $(1-\delta)n \leq \kappa(u)_i \leq (1+\delta)n$, and $u_{\max}$ is bounded by lemma 3.

Furthermore, in order to use the proof idea of \cite{Khachiyan1996}, we need to bound the number of iteration steps of halving $\delta$, when $\delta < \frac{1}{2}$.

\begin{lemma}\quad\label{lemma4}
Suppose $\delta \leq \frac{1}{2}$,
\begin{itemize}
\item If $u$ is not $\delta - $ primal feasible, any ``add-'', ``increase-'' iteration improves $h(u)$ by at least $\frac{2}{5}\delta^2$;
\item If a feasible $u$ does not satisfy the $\delta-$ approximate optimality conditions, any ``decrease -'' iteration improves $h(u)$ by at least $\frac{1}{2} \delta^2$.
\end{itemize}
\end{lemma}
\begin{proof}\quad
From Lemma 2 in the case of the ``add -'' and ``increase -'' iteration,
\[
h(u_k) - h(u_{k+1}) \geq \frac{(n-\kappa(u_k)_{+})^2}{2\kappa(u_k)_{+}^2},
\]
with $\kappa(u_k)_{+}=(1+\epsilon_k) n$, and $\epsilon_k \geq \delta$. Therefore we have:
\[
\begin{array}{ll}
h(u_k) - h(u_{k+1}) & \geq \frac{\epsilon_k^2}{2(1+\epsilon_k)^2},\\
\ & \geq \frac{\delta^2}{2(1+\delta)^2},\\
\ & \geq \frac{2}{5} \delta^2.\\
\end{array}
\]
The second inequality is due to the monotone nondecreasing property of the function $\phi(x)=\frac{x^2}{(1+x)^2}$, and the third inequality is due to the fact $\delta \leq \frac{1}{2}$.

Similarly, in the case of the ``decrease '' iteration, we have:
\[
\begin{array}{ll}
h(u_k) - h(u_{k+1}) & \geq \frac{(n-\kappa(u_k)_{-})^2}{2n\kappa(u_k)_{-}},\\
\ & \geq \frac{\delta^2}{2(1-\delta)},\\
\ & \geq \frac{\delta^2}{2}.\\
\end{array}
\]
\end{proof}

Based on (3) and lemma 4, let $k(\delta)$ denote the number of iterations of the CD algorithm from the first iteration that is $\delta$ approximately optimal until the first that is $\frac{\delta}{2}$ approximately optimal, then :
\[
k(\delta) \leq \frac{n(1+m u_{\max})}{\frac{2(\delta \slash 2)^2}{5}}=\frac{10n(1+mu_{\max})}{\delta}.
\]
Therefore if $K(\epsilon)$ denotes the number of iterations from the first iterate that is 1 approximately optimal until the first that is $\epsilon$ approximately optimal, we find:
\[
\begin{array}{ll}
K(\epsilon) & \leq k(1) + k(\frac{1}{2}) +\cdots+k(\frac{1}{2^{\ulcorner\ln \frac{1}{\epsilon}\urcorner-1}}),\\
\ & \leq \frac{20n(1+m u_{\max})}{\epsilon}.\\
\end{array}
\]
In addition, as for the drop-iterations, similar explanation in the proof of theorem 1 can be applied.

We concludes the above analysis into the theorem below.
\begin{theorem}
The total number of iterations for the CD algorithm with KIS to obtain an $\epsilon$ approximately optimal solution is at most $\frac{40n(1+m u_{\max})}{\epsilon} + \frac{n}{c_1} (\ln\frac{n \ln m}{c_2})+m$; The total number of iterations for the CD algorithm with KYIS to obtain an $\epsilon$ approximately optimal solution is at most $\frac{40n(1+m u_{\max})}{\epsilon} + \frac{n}{c_1} (\ln\frac{5 n \ln n}{c_2})+2n$.
\end{theorem}
\subsection{Local Convergence Rate of the CD algorithm}
We follow the technique in \cite{Ahipasaoglu2008},\cite{Todd2016} to prove the locally linear convergence rate of the CD algorithm. Suppose $u$ satisfies the $\delta$ approximately optimal condition, and $H(u) = (XUX^T)^{-1}$, set
\[
z_i = \left\{
\begin{array}{ll}
n \delta, &\ \textrm{if}\ u_i =0, \\
x_i^T H(u) x_i - n, & \ \textrm{if}\ u_i >0.\\
\end{array}
\right.
\]

Consider the $z-$ perturbation problem in \cite{Ahipasaoglu2008}, we have $u^T z = \sum\limits_{\{i|u_i>0\}} u_i (x_i^T H(u) x_i - n)=n(1-e^T u)$.

\begin{lemma}\quad\label{lemma5}
Suppose $u$ is a $\delta - $ approximately optimal solution, then
\[
h(u) - h^{*} \leq \|u-u^{*}\| \|z\|
\]
\end{lemma}
\begin{proof}\quad
Let $\Phi(z)$ denote the value function, the optimal value of the $z-$ perturbation problem.
\[
\begin{array}{ll}
h(u) & = \Phi(z) + n(e^T u-1),\\
\ & \geq \Phi(0) + (u^{*}- u)^T z,\\
\ & \geq  h^{*} - \|u-u^{*}\| \|z\|.\\
\end{array}
\]
The second inequality is due to the convexity of the function $f(H) = -\ln \det H$, and the second inequality is due to $e^T u^{*}=1$.
\end{proof}

\begin{theorem}\quad
The CD algorithm is locally linearly convergent.
\end{theorem}
\begin{proof}\quad
By lemma 5, there is a constant M such that for every sufficiently small positive $\delta$, any $\delta$ approximately optimal $u$ satisfies
\[
g(u)-g^{*} \geq M \delta^2.
\]
Therefore halving $\delta$ will require at most
\[
\frac{M \delta^2}{\frac{2(\delta\slash 2)^2}{5}}=10M
\]
iterations. Thus adding all these iterations gives a total of $20M\log_2(\frac{1}{\epsilon})$. As a consequent, there exists data dependent constants $P$ and $Q$ such that the number of iterations of the CD algorithm required to obtain an $\epsilon$ approximately optimal $u$ is at most
\[
P+ Q \log_2(\frac{1}{\epsilon}).
\]
\end{proof}

{\bf We note that from the proof of lemma 3, the decrement of $h(u)$ of the CD algorithm between two iteration points $u_k$ and $u_{k+1}$ is larger than that of the WA algorithm, this explains why the CD algorithm slightly outperforms the WA algorithm during the later stage of the iterations. The numerical results in the 5th section supports this theoretical results.}
\subsection{Complexity Analysis of the CD algorithm}
The CD algorithm consists of two stages, the initialization stage and the iteration stage. In the initialization stage, we shall calculate $\kappa_j$ for each $j$ after obtaining the initial point $u_0$ by Kumar-Yildirim's initialization scheme, whose complexity is $O(n^2m)$. In the calculation, first denote $XU_0X^T = A^T A$, where $A^T$ can be factorized as $A^T=Q_0 R_0$ with $Q_0$ and $R_0$ being an orthogonal matrix and a upper-triangle matrix respectively. Then it can be obtained that
\[
\begin{array}{cl}
\kappa_j & = x_j^T (XU_0X^T)^{-1} x_j \\
\ & = x_j^T R_0^{-1} R_0^{-T} x_j \\
\ & = (R_0^{-T} x_j)^{T} (R_0^{-T} x_j),\\
\end{array}
\]
from which we can see that the calculation includes the following:
\begin{enumerate}
\item calculating $A^T = X U_0^{\frac{1}{2}}$ whose complexity is $O(mn)$,
\item the QR factorization of $A^T$ whose complexity is $O(mn^2)$. The reason here using the QR decomposition instead of the Cholesky factorization is mainly due to the numerical stability (see \cite{Todd2016}),
\item calculating $R_0^{-T}$ whose complexity is $O(n^3)$,
\item calculating $R_0^{-T} x_j$ whose complexity is $O(n^2)$,
\item calculating $\kappa_j = (R_0^{-T} x_j)^T (R_0^{-T} x_j)$ whose complexity is $O(n)$.
\end{enumerate}
As there are in total $m$ different $\kappa_j$ s, the total complexity of the initialization part is $O(mn^2)+ O(m(mn+mn^2+n^3+n^2+n)) = O(m^2n^2)$.

During the iterations, the calculations include
\begin{enumerate}
\item finding out the biggest and smallest among the $\kappa_j$,$j=1,\ldots,m$, whose complexity is $O(m)$,
\item comparing $\frac{\kappa_{+} - n}{n}$ and $\frac{n-\kappa_{-}}{n}$, whose complexity is $O(1)$,
\item calculating the step size $\lambda$ and the next iteration point $u^{\dagger}$, whose complexity is $O(1)$,
\item adopting the Cholesky rank-one modification update $XU^{\dagger}X^T = XUX^T - \lambda x_j x_j^T$, whose complexity is $O(mn^2)$.
\end{enumerate}
As the number of iterations is $O(\frac{1}{\epsilon})$, the total complexity of this part is $O(\frac{1}{\epsilon} (m+n^2+mn^2)) = O(\frac{1}{\epsilon} (mn^2))$. Adding these two parts together, the complexity of the CD algorithm is  $O(mn^2 \frac{1}{\epsilon})$.

\section{Numerical Tests}
\subsection{Description of Samples}
Extensive computational tests have been carried out and discussed in several research papers (with reference to \cite{Sun2004}, \cite{Ahipasaoglu2008}). According to the most recent result by \cite{Todd2016}, we know that
\begin{enumerate}
\item The WA algorithm achieves the most accurate solutions.

\item Kumar and Yildirim's initialization scheme truly accelerated the speed of computation.

\item The points elimination strategy used in \cite{Ahipasaoglu2009}) speeds up the computation.

\item On moderately sized samples (dimension $\leq $ 30, number of points $\leq$ 1800), when using the active set strategy, the DRN algorithm outperforms the other algorithms, while without the active set strategy, the WA algorithm outperforms the other algorithms. On large-scale samples (dimension $\leq $ 30, number of points $\leq$ 30,000), the DRN method is superior. on huge-scale sized samples (dimension = 500, number of points =500,000), the WA algorithm is the only feasible algorithm, while the other methods failed because of memory problem.
\end{enumerate}

Based on the conclusion of \cite{Todd2016}, in this section, we focus only on the efficiency of the WA algorithm and the CD algorithm. The aim of this numerical test is to try to determine: under the same accuracy (the error is set to be $10^{-7}$), the efficiency of the WA algorithm and the CD algorithm on various test samples, including small size samples, moderate size samples, large-scale size samples and huge-scale size samples.
\begin{description}
\item [\bf Test samples.] According to \cite{Todd2016}, if the test samples are generated from a general Gaussian distribution, they will all lie close to the surface of an ellipsoid, which is a special situation. Therefore, we use the point generating strategy of \cite{Todd2016} to build up our test sets.

\item [\bf Size.] To compare the efficiencies of the WA algorithm and the CD algorithm (with constant stepsize) on various samples, we test the two algorithms with respect to 4 different sizes of data sets, which are: $n=10, m=500$ for small size test sets, $n=30, m=1800$ for moderate size test sets, $n=100, m=30,000$ for large-scale size test sets and $n=500, m=500,000$ for huge-scale size test sets. Moreover, in order to determine whether merely the dimension influences the efficiency of the CD algorithm and the WA algorithm, we use two more test sets ($n=500, m=1000$ for a high dimension and small size data set (HDSN), while $n=20, m=100,000$ for a low dimension and large size data set (LDLN).

\item [\bf Parameters.] According to Todd (2016), the WA algorithm with Kumar-Yilidirim's initialization scheme can achieve a high accuracy with the error bound being $\epsilon = 10^{-7}$, which is used here. With respect to the initialization scheme, Kumar-Yilidirim's initialization scheme is used in this study, since it is superior both theoretically and practically.
\end{description}
\subsection{Comparison of the CD and WA Algorithm}
The computational results of the CD algorithm and the WA algorithm on different sizes of test sets are illustrated in Table \ref{tab5}. The experiments were performed 10 times for each situation.

With respect to the small size sample, the average iterations and computing time of the WA algorithm outperforms the CD algorithm slightly. More specifically, 7 in 10 cases support the superiority of the WA algorithm, especially in cases 6 and 7, while the remaining 3 cases support the superiority of the CD algorithm. It is noted that the efficiency of the two algorithms on the small scale sample are nearly equivalent to each other, and it follows that with respect to small data sets, these two algorithms are competent. With respect to the moderate size sample, the superiority of the WA algorithm seems to be more apparent, that is, the average iterations of the WA algorithm is 268.6, which is less than 292.9 average iterations of the CD algorithm.
More specifically, 6 in 10 cases support the superiority of the WA algorithm, especially in case 8.
With respect to the large-scale size sample, the average iterations and computing time of the CD algorithm are 731 and 0.52 seconds, respectively, less than 767.1 iterations and 0.57 seconds of the WA algorithm. It is noted that 10 cases all support the superiority of the CD algorithm, and this superiority is much apparent. With respect to the huge-scale size sample, the superiority of the CD algorithm is much more apparent. The results supports that the efficiency of the CD algorithm outperforms that of the WA algorithm as the dimension of the data set is large.

\begin{table}\label{tab5}
\centering
{\footnotesize \begin{tabular}{|c|c|c|c|c|c|c|}
\hline
\multicolumn{1}{|c|}{\multirow {4}{*}{Sample Size}}&\multicolumn{2}{c|}{\multirow {3}{*}{Dimension}}&\multicolumn{4}{c|}{Algorithm}\\
\cline{4-7}
\multicolumn{1}{|c|}{}& \multicolumn{2}{c|}{}& \multicolumn{2}{c|}{CD Algorithm}&\multicolumn{2}{c|}{WA Algorithm} \\
\cline{4-7}
\multicolumn{1}{|c|}{}&\multicolumn{2}{c|}{}&  & Solution Time& &Solution Time\\
\cline{2-3}
 \multicolumn{1}{|c|}{}&\ \  n \ \ &m&Iterations&(seconds)&Iterations&(seconds)\\
\hline
\multirow{10}{*}{Small Size} & 10&500&96&0$^{*}$&88&0\\
& 10&500&34&0&46&0\\
&10&	500&	147&	0.06 &	138&	0.06\\
&10	&500&	386&	0.08 &	375&	0.06\\
&10	&500&	76&	0.06 &	82&	0.06\\
&10	&500&	165&	0.06 &	133&	0.00\\
&10	&500&	147&	0.06 &	91	&0.06\\
&10&	500&	117&	0.06& 	116&	0.06\\
&10&	500&	102	&0.00 &	104&	0.00\\
&10&	500&	110&	0.00 &	81&	0.00\\ \hline
\multicolumn{3}{|c|}{Arithmetic Mean$^{**}$}&	138&	0.04 &	125.4	&0.03\\
\hline
\multirow{10}{*}{Moderate Size} & 30&1800&239&0.39&220&0.47\\
&30&	1800	&234	&0.16	&234	&0.14\\
&30	&1800	&248	&0.16 &	268	&0.12\\
&30	&1800&	236	&0.17 &	265&	0.14\\
&30&	1800&	402&	0.25 &	393&	0.25\\
&30	&1800&	229&	0.14& 	232	&0.09\\
&30&	1800&	285&	0.19&	264&	0.19\\
&30&	1800&	446&	0.11 &	230&	0.11\\
&30	&1800&	246&	0.23 &	233	&0.14\\
&30&	1800&	364&	0.20 &	347	&0.18\\ \hline
 \multicolumn{3}{|c|}{Arithmetic Mean}&	292.9&	0.20 &	268.6&	0.18\\
\hline
\multirow{10}{*}{Large-Scale Size}	&100	&30000&737	&0.53 	&743	&0.53\\
&100	&30000	&658	&0.48 	&710	&0.53\\
&100	&30000	&808	&0.51 	&846	&0.59\\
&100	&30000	&746	&0.56 	&841	&0.66\\
&100	&30000	&713	&0.51 	&780	&0.61\\
&100	&30000	&770	&0.62 	&802	&0.66\\
&100	&30000	&691	&0.51 	&729	&0.56\\
&100	&30000	&661	&0.51 	&680	&0.48\\
&100	&30000	&681	&0.44 	&684	&0.42\\
&100	&30000	&845	&0.53 	&856	&0.62\\ \hline
\multicolumn{3}{|c|}{Arithmetic Mean}	&731	&0.52 	&767.1	&0.57\\
\hline
\multirow{10}{*}{Huge-Scale Size} & 500	&500000	&3081	&63.20 	&3687	&66.64\\
&500	&500000	&3107	&64.21 	&3727	&68.75\\
&500	&500000	&3213	&73.09 	&3663	&76.80\\
&500	&500000	&3171	&64.52 	&3730	&67.39\\
&500	&500000	&3084	&68.56 	&3831	&68.06\\
&500	&500000	&3148	&64.85 	&3691	&67.63\\
&500	&500000	&3096	&67.88 	&3611	&69.58\\
&500	&500000	&3122	&61.28 	&3734	&64.01\\
&500	&500000	&3224	&67.17 	&3829	&74.05\\
&500	&500000	&3096	&66.13 	&3782	&67.45\\
\hline
\multicolumn{3}{|c|}{Arithmetic Mean}	&3134.2	&66.09 	&3728.5	&69.04\\
\hline
\end{tabular}

{\tiny ``*'': the value of computing time is less than $10^{-14}$ seconds.}

{\tiny ``**'': Since there exists ``zero value'', we cannot resort to the geometric mean which is used in Sun and Freud, 2004.}
\caption{Performance of algorithms CD, and WA on small, moderate, large-scale and huge-scale sized problem instances of the minimum volume enclosing ellipsoid problem.}}
\end{table}

Figure 1 shows the progress of $\max \kappa(u)_i$ and $\min\{\kappa(u)_j:u_j >0\}$ under small (subplot 1), moderate (subplot 2), large-scale (subplot 3) and huge-scale(subplot 4) size samples, and the linear convergence of the error $\epsilon$ plotted on a log scale. The result is consistent with that in \cite{Todd2016}, that is, the linear convergence takes hold almost from the first iteration.{\bf Furthermore, we could see that there really exists an inflection point after which the CD algorithm prevails.}

\begin{center}
$\langle$ \bf{Please Insert Figure 1 Approximately Here} $\rangle$
\end{center}

Moreover, we would like to know whether the dimension or the number of points influence the efficiency of the two algorithms. We use two types of test samples (i.e., $n=500, m=1000$ for the HDSN case and $n=20, m=100,000$ for the  LDLN case) to test our assumptions. With respect to the HDSN case in figure 2, the ten subgraphs all show that the CD algorithm outperforms the WA algorithm. However, with respect to the LDLN case in figure 3, the WA algorithm outperforms the CD algorithm in subgraphs 3, 4, 5, 6, 8 and 10, while in four other subgraphs, the CD algorithm is better. This comparison shows that the dimension strongly influences the performance of the algorithms, that is, when the dimension of the data set becomes relatively high, the performance of the WA algorithm becomes worse, while the performance of the CD algorithm becomes better.

\begin{center}
$\langle$ \bf{Please Insert Figure 2 and Figure 3 Approximately Here} $\rangle$
\end{center}

Moreover, we illustrate the reduction of the objective function under three different iteration types, i.e., ``add-'' iteration, ``plus-'' iteration and ``minus-'' iteration. First, during each ``add-'' iteration or ``plus-'' iteration, the reduction of the objective function $h(u)$ is
\[
\begin{array}{cl}
\Delta^{+}(\kappa_{+}) & = h(u) - h(u^{\dagger}) \\
\ & = \ln (2 \kappa_{+} - n) - \ln (\kappa_{+}) + \frac{(n-\kappa_{+})n}{\kappa_{+}^2}, \\
\end{array}
\]
where $\kappa_{+} \geq n$. Second, during each ``minus-'' iteration step, the reduction of the objective function $h(u)$ is:
\[
\begin{array}{cl}
\Delta^{-} (\kappa_{-}) & = h(u) - h(u^{\dagger}) \\
\ & = \ln \frac{\kappa_{-}}{n} + \frac{n}{\kappa_{-}} -1 \\
\end{array}
\]
where $0 \leq \kappa_{-} \leq n$. We illustrate the graph of $\Delta^{+}(\kappa_{+})$ and $\Delta^{-}(\kappa_{-})$ in Figure 4. In Figure 4, the horizontal axis is the value of $\kappa$ and the vertical axis is the amount of decrement of $h(u)$. We consider three cases of $n$ as $n=1$, $n=2$  and $n=3$, illustrated as blue, green and red lines respectively in Figure 4. The solid line represents the ``add-'' and ``plus-'' iteration (namely, the graph of $\Delta^{+}(\kappa_{+})$), and the dashed line represents the ``minus-'' iteration (namely, the graph of $\Delta^{-}(\kappa_{-})$). From Figure 5, we can see the following properties:

(1) $\Delta^{-}(\kappa_{-})$ is monotone decreasing with respect to $\kappa_{-}$, while $\Delta^{+}(\kappa_{+})$ is monotone increasing with respect to $\kappa_{+}$.

(2) $\lim\limits_{\kappa_{+} \rightarrow \infty} \Delta^{+}(\kappa_{+}) = \ln 2$¡£

(3) When $0 \leq \kappa_{-} \leq 0.3n$, $\Delta^{-}(\kappa_{-}) > \ln 2$. In other words, when $0 \leq \kappa_{-} \leq 0.3n$, $\Delta^{-}(\kappa_{-})$ is always larger than $\Delta^{+}(\kappa_{+})$.

\begin{center}
$\langle$ \bf{Please Insert Figure 4 Approximately Here} $\rangle$
\end{center}

\section{More Versions of the CD Algorithm}
In the previous section we considered the ``constant'' stepsize. In this section, we investigate more on the CD algorithm, i.e., first we consider more types of stepsizes and then we compare the randomized coordinate descent algorithm on the MVEE problem with the CD algorithm.
\subsection{Alternative Stepsizes}
We discuss two alternative stepsizes, i.e., the diminishing stepsize and the backtracking line search stepsize.

First we introduce the CD algorithm with the diminishing stepsize. The diminishing stepsize used in the algorithm is depicted as follows:
\begin{algorithm}[H]
\caption{CD algorithm with Diminishing Stepsize for model ($D_1$)}\label{4thAlgorithm}
\begin{algorithmic}[1]
\IF {$\frac{\kappa_{+}-n}{n}> \frac{n-\kappa_{-}}{n}$}
\STATE {$\lambda=\frac{2}{k+2}, u_{k+1}=u_k + \lambda e_{j+}$}
\ELSE
\STATE {$\lambda=\max\{-u_{j-},-\frac{2}{k+2}\},u_{k+1}=u_k + \lambda e_{j-}$}
\ENDIF
\end{algorithmic}
\end{algorithm}

Since the diminishing stepsize can be defined as ``$\{\lambda_k | \sum\limits_{k=1}^{\infty} \lambda_k = \infty, \lim_{k \rightarrow \infty} \lambda_k = 0\}$'', in the algorithm above, we take $\lambda_k = \frac{2}{k+2}$ as a special case. The convergence of the CD algorithm with diminishing stepsize is easy to prove. Since the objective function $h(u)$ is ``algorithmic'' coordinate-wise smooth w.r.t the CD algorithm, i.e., $L_i$ - coordinate wise smooth, where $L_i = \kappa(u)_i^2$ (when $\kappa(u)_i \leq n$) or $n \kappa(u)_i$ (when $\kappa(u)_i < n$), therefore the stepsize chosen in the interval $(0,\frac{2}{L_i})$ is suitable for convergence. Moreover, the diminishing stepsize $\lambda_k \rightarrow 0$ with $k \rightarrow \infty$, therefore there exists a $\bar{k}$ such that $\forall k > \bar{k}$, $\lambda_k \in (0,\frac{2}{L_i})$, and the convergence is guaranteed.
Obviously the diminishing stepsize is becoming arbitrarily small during the iteration, therefore we prefer to use the ``constant'' stepsize.

Then we discuss the backtracking line search stepsize, which is defined as:
\begin{algorithm}[H]
\caption{Backtracking Line Search Stepsize}\label{5thAlgorithm}
\begin{algorithmic}[1]
\STATE {Input: An iterative direction $e_i$, $\alpha=0.5$, $\beta \in (0,1)$}
\WHILE {$h(u+\lambda e_i) - h(u) > \alpha \lambda \kappa(u)_i$}
\STATE {$\lambda=\beta \lambda$}
\ENDWHILE
\end{algorithmic}
\end{algorithm}

The convergence of the backtracking line search can be obtained similarly to the diminishing stepsize. It is noted that the backtracking line search is also slow than the ``constant'' stepsize. Since during the early stage of the iteration, the value of $\kappa(u)_i, i=1,\ldots,m$ is far from $n$, and the interval to guarantee convergence is relatively small. However, as the iteration grows, $\kappa(u)_i,i=1,\ldots,m$ is getting close to $n$, and the interval becomes relatively large. Therefore, the backtracking line search strategy determines the stepsize at the early stage of the iteration, and will not change during rest the iteration. However, as the iteration grows, the ``constant'' stepsize $\frac{1}{\kappa(u)_{j+}^2}$ or $\frac{1}{n\kappa(u)_{j-}}$ is becoming large. Therefore the CD algorithm with ``constant'' stepsize is becoming faster during the iteration. Another deficiency of the backtracking line search is that it needs to compute the value of $h(u_k)$ at each iteration.
\subsection{Randomized Coordinate Descent Method}
The CD algorithm proposed in this paper is much similar to the algorithm (algorithm (1.3) on page 1) in \cite{Nesterov2011}. Nesterov pointed out that ``the CD algorithm requires computation of the whole gradient vector. However, if this vector is available, it seems better to apply the usual full-gradient methods''. However as we pointed previously, the full gradient method may be inefficient (even not applicable) in computing the MVEE, and therefore the MVEE problem may be a good example to show the wide application prospect of the coordinate descent method.

In \cite{Nesterov2011}, Nesterov also proposed the randomized coordinate descent method (RCDM), whose coordinate axis selection rule is random, and it is proved to be more efficient than the ``non-randomized selection rule'' (such as the ``Gauss - Southwell's Rule''), since randomized selection rule may not require to compute the full gradient during each iteration. However with respect to the MVEE problem, this is not the case. From Lemma 1 we know that the objective function $h(u)$ of model ($D_1$) is ``algorithmic coordinate-wise smooth'', therefore we need the full gradient information so as to determine the ``random counter'' (methods to choose coordinate axis randomly). The random counter $R$ proposed here is slightly different from that in \cite{Nesterov2011}. It generates an integer number $i \in \{1, \ldots, m\}$ with probability
\[
p_i = |\nabla_i h(u)| (\sum\limits_{j=1}^m |\nabla_j h(u)|)^{-1}, i=1,\ldots,m.
\]
Based on the work of \cite{Nesterov2011}, we propose the randomized coordinate descent (RCD) algorithm here to address the MVEE as follows:
\begin{algorithm}[H]
\caption{RCD algorithm to compute the MVEE}\label{6thAlgorithm}
\begin{algorithmic}[1]
\STATE {Input:$X=\{x_1, \ldots, x_m\}$, $\epsilon > 0$}
\STATE {Initialization: Using Khachiyan's or Kumar - Yildirim's initialize scheme to initialize $u_0$}
\WHILE {not converged}
\STATE {Choose $j = R$}
\STATE {$u_{k+1}=P_{u \geq 0}(u_k + \frac{1}{L_j}\nabla h(u_k)_j)$}
\ENDWHILE
\STATE {Output: $u^{*}=u_k$}
\end{algorithmic}
\end{algorithm}
It is obviously that the the random counter $R$ in \cite{Nesterov2011} applies the ``Randomized Nesterov's Rule'' while $R$ in this paper applies the ``Randomized Gauss - Southwell's Rule''. The convergence results of the RCD algorithm can be referred to the \cite{Nesterov2011}. However the numerical tests shows that the solution of the RCD algorithm is less accurate than the CD algorithm.
\subsection{Computational Performance of the CD Algorithm on Various Stepsizes}
Table \ref{tab4} shows the performance of the CD algorithm of various stepsizes.

\begin{table}\label{tab4}
\centering
{\footnotesize \begin{tabular}{|c|c|c|c|c|c|}
\hline
\multicolumn{1}{|c|}{\multirow {4}{*}{Sample Size}}&\multicolumn{1}{c|}{\multirow {3}{*}{Error}}&\multicolumn{4}{c|}{CD Algorithm}\\
\cline{3-6}
\multicolumn{1}{|c|}{}& \multicolumn{1}{c|}{}& \multicolumn{2}{c|}{``Constant'' Stepsize}&\multicolumn{2}{c|}{Diminishing Stepsize} \\
\cline{2-6}
\multicolumn{1}{|c|}{}&\multicolumn{1}{c|}{}&  & Solution Time& &Solution Time\\
\multicolumn{1}{|c|}{}&$\epsilon$ &Iterations&(seconds)&Iterations&(seconds)\\
\hline
\multirow{3}{*}{Small Size} & $10^{-2}$&9&0.03&62&0.05\\
& $10^{-3}$&31&$0.00^{*}$&10002&1.39\\
&$10^{-4}$&72&0.03&	$100000^{**}$&19.64\\
\hline
\multirow{3}{*}{Moderate Size} & $10^{-2}$&81&0.05&72&0.03\\
& $10^{-3}$&118&0.06&69&0.03\\
&$10^{-4}$&209&0.11&$100000^{**}$&47.88\\
\hline
\multirow{3}{*}{Large-Scale Size} & $10^{-2}$&229&2.26&957& 6.77\\
& $10^{-3}$&313&2.48&2853&19.94\\
&$10^{-4}$&475&3.88&$100000^{**}$&696.53\\
\hline
\end{tabular}}

{\tiny ``*'': the value of computing time is less than $10^{-14}$ seconds.}

{\tiny ``**'': max iteration steps have been reached.}
\caption{Performance of algorithms CD with ``constant'' and diminishing stepsizes on small, moderate, large-scale sized problem instances of the minimum volume enclosing ellipsoid problem.}
\end{table}

In table 3, the error bounds are set to be $10^{-2}$, $10^{-3}$ and $10^{-4}$, which means accuracy of the solution is not high. With respect to the ``constant'' stepsize, the CD algorithm converges fast in all samples. However the error bound made by the diminishing stepsize cannot be reduced to $10^{-4}$. The numerical test is consistent with our theoretical analysis on the stepsize.
\subsection{Comparison of the CD and RCD algorithm}
Figure 5 shows the computational performance of the CD and RCD algorithm. We only test two samples, the small sized sample ($n=10,m=500$) in subplot 1 and 2, the moderate sized sample ($n=30,m=1800$) in subplot 3 and 4. In all these 4 subplots, the horizontal axis represents the number of iterations, and the vertical axis represents the $\log$ of the error bound. With respect to the small sized sample, subplot 1 shows the convergence result of the CD algorithm, and the error bound reduced below $10^{-7}$ after 130 iteration steps. Comparatively, subplot 2 shows the convergence result of the RCD algorithm, and the error bound cannot reduced below $10^{-3}$ in $1 \times 10^4$ (maximum iteration steps ) iteration steps. With respect to the moderate sized sample, subplot 3 shows the convergence result of the CD algorithm, and the error bound reduced below $10^{-7}$ after 350 iteration steps. Comparatively, subplot 4 shows the convergence result of the RCD algorithm, and the error bound cannot reduced below $10^{-2}$ in $1 \times 10^4$ iteration steps. Moreover the ``convergence curve'' of the RCD algorithm is ladder shaped, which means the convergence rate of the RCD algorithm is very slow. Therefore the CD algorithm with ``Gauss - Southwell's Rule'' is much efficient than RCD algorithms at least in the MVEE problem. This result is inconsistent with the theoretical results in \cite{Nesterov2011}, but consistent with the results in \cite{Nutini2015}.
\begin{center}
$\langle$ \bf{Please Insert Figure 5 Approximately Here} $\rangle$
\end{center}
\section{Conclusions}
This paper proposed the coordinate descent algorithm with Gauss - Southwell's Rule, which complements the widely used WA algorithm to address the Minimum Volume Enclosing Ellipsoid problem. In particular, the CD algorithm is slightly more efficient when the dimension of the data set increases. Since the MVEE problem is the basis of many practical problems and applications, such as nonlinear support vector machines and optimal design, the CD algorithm can be seen as a benefit for solving real-world problems. Specifically, the theoretical contribution of this paper is four-fold:
\begin{enumerate}
\item The proposition of a new efficient algorithm: the CD algorithm. This algorithm inherits the property of sublinear convergence and behaves slightly more efficient, especially for the high dimensional data sets;

\item The discovery of some new properties of the objective function of the MVEE problem. We proved the ``algorithmic'' coordinate smoothness of the objective function w.r.t the CD algorithm.

\item The disclosure of the relationship of the first-order oracle algorithm and the linear optimization oracle algorithm in the MVEE problem. Based on the new properties of the objective function discovered in this paper, we found that the FW-K algorithm uses the ``Nesterov's Rule'' to choose coordinate axis, and the WA algorithm uses the ``Gauss - Southwell's Rule''. This finding connects the first-order oracle algorithm with the linear optimization oracle algorithm. Moreover this finding explains the reason why the WA algorithm is more efficient than the FW-K algorithm.

\item The application of the MVEE to address the big data DEA problem. We unified the dual simplex method, sensitivity analysis theory in linear programming, and MVEE together to design a unified MVEE-perturbation algorithm which can solve the big data DEA problem efficiently. Since DEA is widely used in economic and management science, this algorithm would be useful in real-world application.
\end{enumerate}

We also note two directions for future research based on this work. The first point is that we did not take the DRN algorithm into account, since many researchers note that it is not applicable with respect to huge-scale data sets. The reason is that the DRN algorithm resorts to the Hessian of the objective function at each iteration and memory problems arise when the dimension grows. Recently, the Newton-sketch technique was proposed to reduce the dimension of the data sets, which achieved a super-linear convergence rate (see \cite{PW2015}). Therefore the combination of the DRN algorithm and the Newton-sketch technique is a good point for future research. The other point is that \cite{Lu2016} found the dual problem of the MVEE is $1$-smooth relative to $-\sum\limits_{j=1}^n \ln x_j$, and presented a new framework to investigate the MVEE problem. The work of \cite{Lu2016} extended the notion of smoothness and strong convexity, which can be applied to many more problems than the classical framework.

\section{Appendix}
\begin{proposition}\quad
(1) When $\kappa_{+} \geq n > 0$, we have:
\[
\ln(2-\frac{n}{\kappa_{+}}) - \frac{n(\kappa_{+} - n)}{\kappa_{+}^2} - \frac{(n-\kappa_{+})^2}{2\kappa_{+}^2} \geq 0;
\]

(2) When $0 \leq \kappa_{-} \leq n$, we have:
\[
\ln(\frac{\kappa_{-}}{n}) + \frac{n}{\kappa_{-}} - 1 - \frac{(n-\kappa_{-})^2}{2n\kappa_{-}} \geq 0;
\]

(3) When $\kappa_{-} \leq n$, and $\frac{\kappa_{-}-n}{n\kappa_{-}} < - u_{j-} \leq 0$, we have:
\[
n u_{j-} + \ln (1-u_{j-} \kappa_{j-}) \geq 0
\]
\end{proposition}
\begin{proof}\quad
(1) When $t \geq n$, we have $g_1(t) = \ln(2-\frac{n}{t}) - \frac{n(t - n)}{t^2} - \frac{(n-t)^2}{2t^2}$, and $g_1'(t)=-\frac{n}{t}(\frac{(n-t)^2}{t^2 (n-2t)}) \geq 0$. Therefore, $g_1(t)$ is monotone increasing with $t \geq n$. Furthermore, since $g_1(n)=0$, we have $g_1(t) \geq 0$ when $t \geq n$ and Proposition (1) follows.

(2) When $t \leq n$, we have $g_2(t) = \ln(\frac{t}{n}) + \frac{n}{t} - 1 - \frac{(n-t)^2}{2nt}$ and $g_2'(t) = -\frac{(n-t)^2}{2nt^2} \leq 0$. Therefore, $g_2(t)$ is monotone decreasing with $t \leq n$. Moreover, since $g_2(n)=0$, thus when $t \leq n$, we have $g_2(t) \geq 0$ and Proposition (2) follows.

(3) When $0<t<\frac{n-\kappa_{-}}{n\kappa_{-}}$ and $\kappa_{-} < n$, we have $g_3(t) = n t + \ln (1- t \kappa_{-})$ and $g_3'(t) = n- \frac{\kappa_{-}}{1-\kappa_{-}t}$. Using $t < \frac{n-\kappa_{-}}{n\kappa_{-}}$ we have $g_3'(t) \leq 0$. Therefore, $g_3(t)$ is monotone decreasing with $t > 0$. Moreover, since $g_3(0)=0$,  we have $g_3(t) \geq 0$ when $t > 0$ and Proposition (3) follows.
\end{proof}
\Acknowledgements{This work was supported by National Natural Science Foundation of China (Grant No. 71601117).}



\begin{thebibliography}{99}
\bahao\baselineskip 11.5pt

\bibitem{Ahipasaoglu2008} Ahipasaoglu, D. S., Sun, P., Todd, M. J.
Linear convergence of a modified Frank ¨C Wolfe algorithm for computing minimum-volume enclosing ellipsoids.
{Optimization Methods \& Software},
2008, 23(1):5-19

\bibitem{Ahipasaoglu2009}Ahipasaoglu, S. D.
Solving Ellipsoidal Inclusion And Optimal Experimental Design Problems: Theory And Algorithms.
{Cornell University},
2009

\bibitem{Ahipasaoglu2015} Ahipasaoglu, S. D.
Fast algorithms for the minimum volume estimator.
{Journal of Global Optimization},
2015, 62(2):351-370

\bibitem{Ali1994}Ali, A. I.
Computational aspects of DEA.
Data Envelopment Analysis: Theory, Methodology, and Applications.
{Springer, Netherlands}
1994: 63-88

\bibitem{Atwood1973} Atwood, C. L.,
Sequences converging to D - optimal designs of experiments.
{the Annals of Statistics},
1973, 342-352

\bibitem{Bertsekas2015} Bertsekas, D. P.
Convex optimization algorithms.
{Athena Scientific, Belmont, MA},
2015

\bibitem{Boyd2004} Boyd, S., Vandenberghe, L.
Convex Optimization.
{Cambridge university press, Cambridge},
2004

\bibitem{Bubeck2015} Bubeck, S.
Convex optimization: Algorithms and complexity.
{Foundations and Trends in Machine Learning}
2015, 8(3-4): 231-357

\bibitem{Clarkson2010} Clarkson, K. L.
Coresets, sparse greedy approximation, and the Frank-Wolfe algorithm.
{ACM Transactions on Algorithms},
 2010, 6(4): 63

\bibitem{Fleming2017} Haraldsd$\acute{o}$ttir, et al.
CHRR: coordinate hit-and-run with rounding for uniform sampling of constraint-based models.
{Bioinformatics},
2017, 33.11: 1741-1743

\bibitem{John2014}John, F.
Extremum problems with inequalities as subsidiary conditions.
{Springer, Basel},
2014

\bibitem{Joshi2009} Joshi, S., Boyd, S.
Sensor selection via convex optimization.
{IEEE Transaction on Signal Process.}
2009, 57(2):451-462

\bibitem{Khachiyan1996} Khachiyan L G (1996)
Rounding of polytopes in the real number model of computation.
{Mathematics of Operations Research},
1996, 21(2):307-320

\bibitem{Kumar2005} Kumar, P., Yildirim, E. A.
Minimum - volume enclosing ellipsoids and core sets.
{Journal of Optimization Theory and Application},
2005, 126(1):1-21

\bibitem{Kumar2011} Kumar, P., Yildirim, E. A.,
A linearly convergent linear - time first - order algorithm for support vector classification with a core set result.
{Informs Journal on Computing},
2011, 23(3):377-391

\bibitem{Li2015} Li, Y., Wang, C., Shene, C. K.
Extracting flow features via supervised streamline segmentation.
{Computers \& Graphics},
2015, 52:79-92

\bibitem{Lu2016} Lu, H., Freund, R. M., Nesterov, Y.
Relatively-Smooth Convex Optimization by First-Order Methods, and Applications.
{SIAM Journal on Optimization},
2018, 28.1: 333-354

\bibitem{LuoTseng1992} Luo, Z. Q., Tseng, P.
On the linear convergence of descent methods for convex essentially smooth minimization.
{SIAM Journal on Control and Optimization.}
1992, 30(2): 408-425

\bibitem{Nesterov1994} Nesterov, Y., Nemirovskii, A.
Interior-point polynomial algorithms in convex programming.
{Society for Industrial and Applied Mathematics, Philadelphia},
1994

\bibitem{Nesterov2004} Nesterov, Y.
Introductory Lectures on Convex Optimization: A Basic Course.
{Kluwer Academic Publishers, Boston},
2004

\bibitem{Nesterov2011} Nesterov, Y.
Efficiency of coordinate descent methods on huge - scale optimization problems.
{SIAM Journal on Optimization},
2011, 22(2): 341-362

\bibitem{Nikolov2016}Nikolov, A., Talwar, K., Zhang, L.
The Geometry of Differential Privacy: The Small Database and Approximate Cases.
{SIAM Journal on Computing},
2016, 45(2):575-616

\bibitem{Nutini2015} Nutini, J., Schmidt, M., Laradji, I.
Coordinate descent converges faster with the Gauss-Southwell rule than random selection.
{International Conference on Machine Learning},
2015: 1632-1641

\bibitem{PW2015} Pilanci M, Wainwright M J
Newton sketch: A linear-time optimization algorithm with linear-quadratic convergence.
{SIAM Journal on Optimization},
2017, 27.1:205-245

\bibitem{Robinson1982} Robinson, S. M.
Generalized equations and their solutions, Part II: Applications to nonlinear programming.
{Optimality and Stability in Mathematical Programming},
1982: 200-221

\bibitem{Sun2004} Sun, P., Freund, R. M.
Computation of minimum - volume covering ellipsoids.
{Operations Research},
2004, 52(5):690-706

\bibitem{Todd2007} Todd, M. J., Yildirim, E. A. (2007)
On Khachiyan's algorithm for the computation of minimum - volume enclosing ellipsoids.
{Discrete Applied Mathematics},
2007, 155(13):1731-1744

\bibitem{Todd2016} Todd, M. J.
Minimum-volume ellipsoids: Theory and algorithms.
{SIAM, Philadelphia},
2016

\bibitem{Vandenberghe1998} Vandenberghe, L., Boyd, S., Wu, S. P.
Determinant maximization with linear matrix inequality constraints.
{SIAM Journal on Matrix Analysis and Application},
1998, 19(2):499-533

\bibitem{Cong2012} Wei, C., et al.
Rank - two update algorithms for the minimum volume enclosing ellipsoid problem.
{Computational Optimization and Applications},
2012, 51(1): 241-257

\bibitem{Cong2017} Wei, C., Lei, H.
Active-Set Algorithm for Computing MVEE Based on New Initilization Strategy (In Chinese)
{Journal of Jilin University (Science Edition)},
2017, 5(55): 1141-1145

\bibitem{Yildirim2006} Yildirim, E. A.
On the minimum volume covering ellipsoid of ellipsoids.
{SIAM Journal on Optimization},
2006, 17(3):621-641
\end{thebibliography}
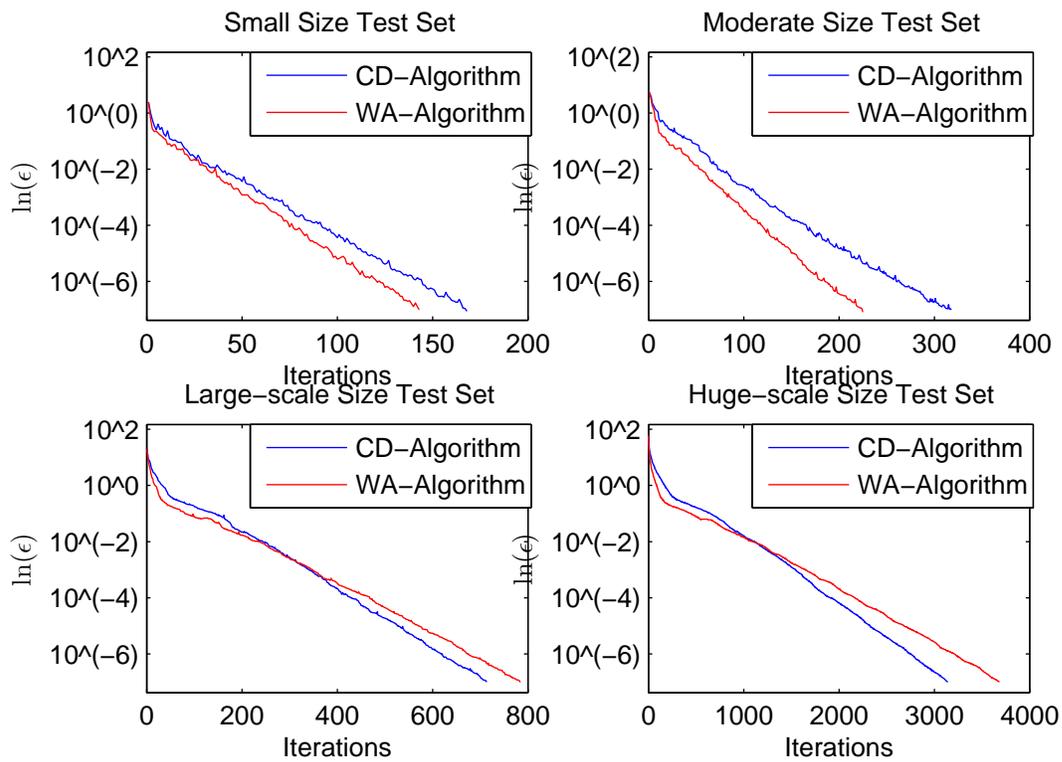
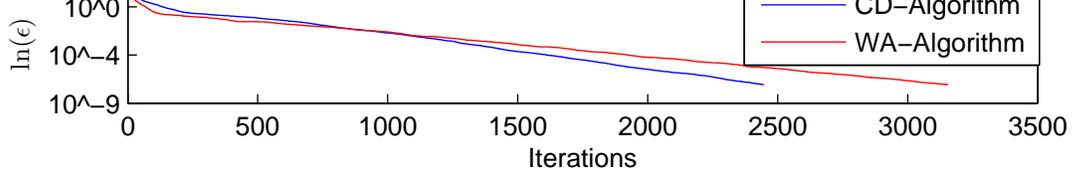
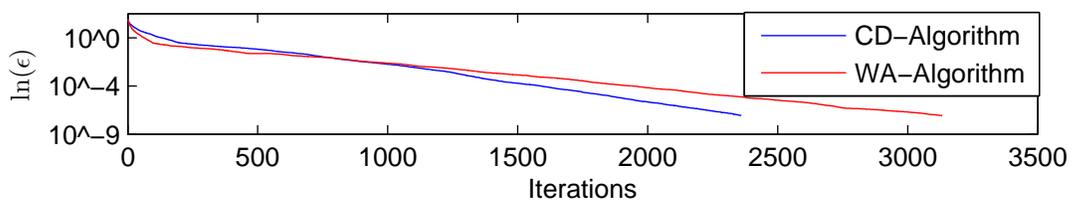
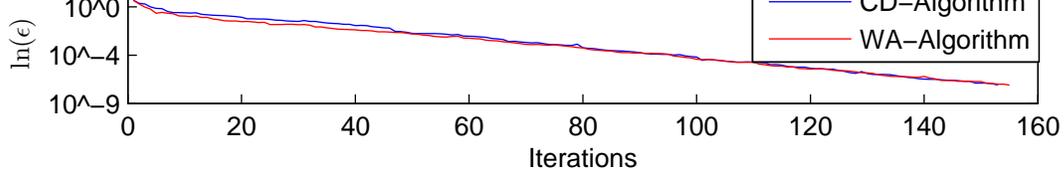
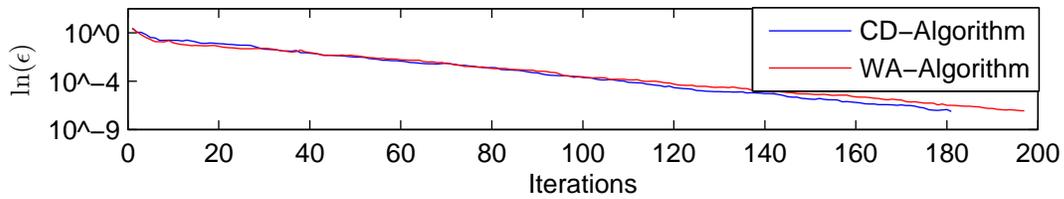
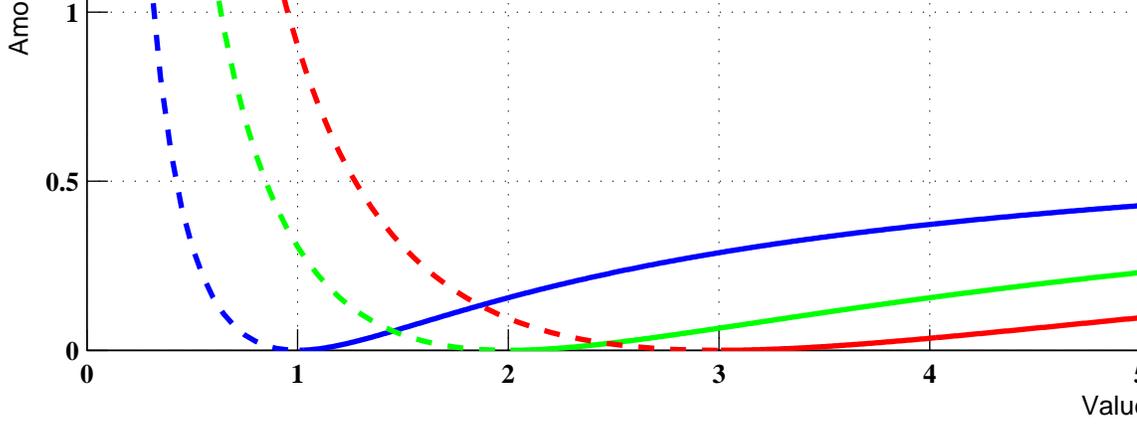
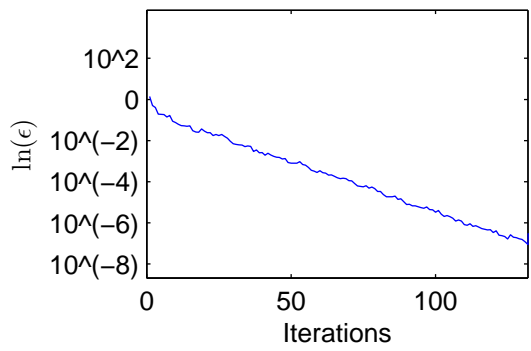
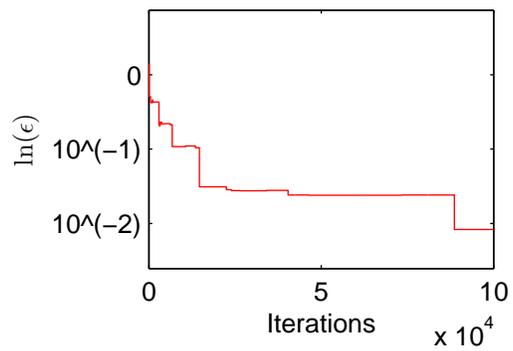
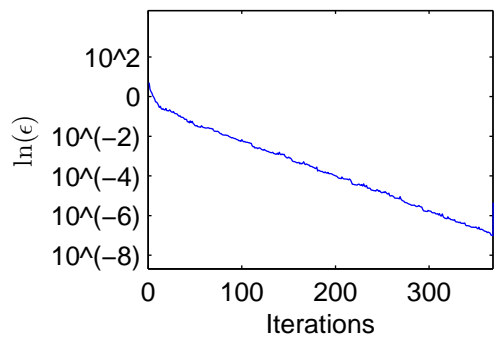
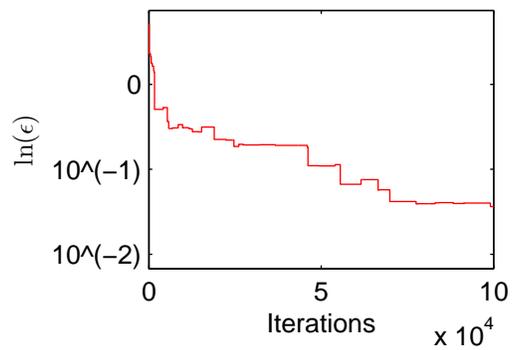
\end{document}